\documentclass[12pt]{article}
\usepackage{amsmath, amssymb,amsfonts,amsthm}
\usepackage{tikz,graphicx,float}
\usepackage{comment}
\usetikzlibrary{arrows}
\tikzset{>=stealth}
\usepackage{geometry}
\usepackage{algorithm}
\usepackage[noend]{algorithmic}
\usepackage{caption}
\newtheorem{theorem}{Theorem}[section]

\newtheorem{definition}[theorem]{Definition}
\newtheorem{prop}[theorem]{Proposition}

\newtheorem{cor}[theorem]{Corollary}
\newtheorem{lemma}[theorem]{Lemma}
\newcommand{\N}{\mathbb{N}}

\begin{document}
\title{Pr\"{u}fer codes on vertex-colored rooted trees}
\author{R.~W.~R.~Darling\thanks{National Security Agency, USA} \and Grant Fickes\\
}
\date{\today}

\maketitle
\abstract{Pr\"{u}fer codes provide an encoding scheme for representing a vertex-labeled tree on $n$ vertices with a string of length $n-2$. Indeed, two labeled trees are isomorphic if and only if their Pr\"{u}fer codes are identical, and this supplies a proof of Cayley's Theorem. Motivated by a graph decomposition of freight networks into a corpus of vertex-colored rooted trees, we extend the notion of Pr\"{u}fer codes to that setting, i.e., trees without a unique labeling, by defining a canonical label for a vertex-colored rooted tree and incorporating vertex colors into our variation of the Pr\"{u}fer code. Given a pair of trees, we prove properties of the vertex-colored Pr\"{u}fer code (abbreviated VCPC) equivalent to (1) isomorphism between a pair of vertex-colored rooted trees, (2) the subtree relationship between vertex-colored rooted trees, and (3) when one vertex-colored rooted tree is isomorphic to a minor of another vertex-colored rooted tree. }


\section{Introduction}

\subsection{Motivation}

Graph decompositions partition a graph into smaller pieces, enabling more effective local study. Various standard graph decomposition algorithms exist, which include breaking the graph into modules (\cite{Gal1967}), cycles (\cite{Veb1912}), Hamiltonian cycles (\cite{Luc1882} attributes the first construction to Walecki), and paths (\cite{RobSey1983}). More abstractly, $(G,\mathcal{F})$-designs (see \cite{ChuGra1981} for a survey) determine if the graph $G$ can be decomposed into components, each of which is contained in a specific family of graphs $\mathcal{F}$. The work of Awerbuch et al. (\cite{Awe1989}) introduced graph decompositions via a ruling set $R$ of vertices so that each component of the resulting decomposition contains at most one representative of $R$. Once the ruling set $R$ is chosen, there are many ways to construct such a decomposition. Techniques of \cite{Awe1989} strive to minimize the average diameter of the resulting components, working towards a parallelizable algorithm. Darling et al. (\cite{DarHarPhuPro2018}) view the multiway cut problem on an edge-weighted forest as a search for a maximum weight basis of a matroid. Instead of decomposing based on the diameter of the component, a maximum weight spanning forest is formed in which no two ruling vertices can co-occur in the same tree. Motivated by freight data (\cite{FreightData}) detailing the total weight and value of top products shipped between major American cities, we seek to push their analysis further. We wish to determine commonly occurring and infrequently occurring structure among the representatives of a ``large'' corpus of graphs. 

\subsection{Literature}

First introduced by Heinz Pr\"{u}fer (\cite{Pru1918}) in 1918 to prove Cayley's Formula (\cite{Cay1889}) for the number of labeled trees on $n$ vertices, Pr\"{u}fer codes provide a bijection between the set of vertex-labeled trees on $n$ vertices and the set $\{0,\dotsc,n-1\}^{n-2}$. Since then, variations on the original definition have been studied by Neville (\cite{Nev1953} - who actually presented three different labeled tree encodings that each fix a root), Moon (\cite{Moo1970} - who proposed rootless variations of Neville's schemes), Caminiti and Petreschi (\cite{CamPet2005}), and Deo and Micikevi\v{c}ius (\cite{DeoMic2002}). However, each of these schemes apply to labeled trees, requiring unique vertex labels from the set $\{0, \dotsc, n-1\}$ for a tree on $n$ vertices. 

There are two obvious problems that arise when applying typical coding schemes to vertex-colored trees, namely (1) how can the color of leaves be recovered and (2) if two vertices have the same color $c$, how can neighbors of the two vertices be distinguished? Pr\"{u}fer codes for edge-colored trees were introduced by Cho et al. (\cite{ChoKimSeoShi2004}) in counting the number of trees on $n$ vertices with $k$-colored edges, a result first obtained by Stanley (\cite{Sta1999}). The construction in \cite{ChoKimSeoShi2004} still assumes a vertex-labeling and creates an array of dimensions $2 \times (n-1)$ to encode a tree on $n$ vertices. The following provides a description of the construction. 
\begin{itemize}
\item First row:
\begin{itemize}
\item The standard Pr\"{u}fer code in the first $n-2$ entries
\item Complete one additional pruning step to populate the ($n-1$)st entry. 
\end{itemize}
\item Second row: The color of the edge removed at the appropriate pruning step. 
\end{itemize}
This encoding of an edge-colored tree motivates our application to vertex-colored trees introduced in Definition \ref{Def:ColoredPruferCodes}. 

\subsection{Isomorphism Problems}

The equality of Pr\"{u}fer (and related) codes for two labeled trees is equivalent to (labeled) isomorphism between them. Once computed, Pr\"{u}fer codes provide a ``fast'' isomorphism check. Since Pr\"{u}fer codes can be computed in $O(n)$ time (\cite{WanWanWu2009}), isomorphism between two labeled trees can be determined in linear time simply by comparing their respective Pr\"{u}fer codes. Other string encodings of trees have also been used to determine isomorphism for rooted trees. Typically, these strings use a sequence of ``('' and ``)'' (or some other type of open/closed brackets) to encode the tree structure. Furthermore, some definitions impose a canonical ordering on the vertices (see \cite{Bus1997} for an example) to standardize the representation. 

Implications of any of these encodings on the subtree isomorphism problem have received little attention. The work of Abboud et al. (\cite{AbbBacHanWilZam2015}) begins with a brief summarization of known algorithms to answer subtree isomorphism in specific settings. However, these do not cover the questions we consider here. The authors distinguish between the cases of subtree isomorphism for rooted and unrooted trees. In the setting of our study, we consider rooted trees, however, given two rooted trees $T$ and $T'$, when determining if one is isomorphic to a subtree of the other, we do not require the root of $T$ to be identified with the root of $T'$. Instead, for $T'$ to be isomorphic to a subtree $S$ of $T$, we merely require the parent-child relationships between vertices of $T'$ to agree with those of $S$, where these vertex-relationships in $S$ are inherited from $T$ (a seemingly unstudied blend of the standard notions of unrooted and rooted isomorphism). As an example, consider the two trees presented in the Figure \ref{Fig:IsomorphicSubtreeDef}. We consider the tree $T'$ on the left isomorphic to a subtree of $T$ (on the right), where the root of $T$ is colored blue, the subtree of $T$ (called $S$) realizing the isomorphism has edges colored red, and the roots of $T'$ and $S$ are colored green. 
\begin{figure}[H]
\centering
\begin{minipage}{0.35\textwidth}
\centering
\begin{tikzpicture}
\draw[thick,black] (-2,-2) -- (-1.5,-1) -- (-1,-2);
\draw[thick,black] (1.5,-2) -- (1.5,-1) -- (0,0);
\draw[thick,black] (-1.5,-1) -- (0,0) -- (0,-1);

\filldraw[green](0,0)circle(0.1);
\filldraw[black](-1.5,-1)circle(0.1);
\filldraw[black](0,-1)circle(0.1);
\filldraw[black](1.5,-1)circle(0.1);
\filldraw[black](-2,-2)circle(0.1);
\filldraw[black](-1,-2)circle(0.1);
\filldraw[black](1.5,-2)circle(0.1);

\node at (-2,0) {$T'$};
\end{tikzpicture}
\end{minipage}
\begin{minipage}{0.55\textwidth}
\centering
\begin{tikzpicture}
\draw[thick,red,dashed] (-2,-3) -- (-1.5,-2) -- (-1,-3);
\draw[thick,red,dashed] (2,-3) -- (1.5,-2) -- (0,-1);
\draw[thick,red,dashed] (-1.5,-2) -- (0,-1) -- (0,-2);
\draw[thick,black] (1.5,-3) -- (1.5,-2) -- (1,-3);

\draw[thick,black] (0,-1) -- (2,0) -- (4,-1) -- (5,-2);
\draw[thick,black] (4,-1) -- (3,-2);

\filldraw[green](0,-1)circle(0.1);
\filldraw[black](-1.5,-2)circle(0.1);
\filldraw[black](0,-2)circle(0.1);
\filldraw[black](1.5,-2)circle(0.1);
\filldraw[black](-2,-3)circle(0.1);
\filldraw[black](-1,-3)circle(0.1);
\filldraw[black](1.5,-3)circle(0.1);
\filldraw[black](1,-3)circle(0.1);
\filldraw[black](2,-3)circle(0.1);

\filldraw[blue](2,0)circle(0.1);
\filldraw[black](4,-1)circle(0.1);
\filldraw[black](3,-2)circle(0.1);
\filldraw[black](5,-2)circle(0.1);

\node at (-1.5,-1) {\textcolor{red}{$S$}};
\node at (4,0) {$T$};
\end{tikzpicture}
\end{minipage}
\caption{\textit{$T'$ is isomorphic to a subtree ($S$) of $T$ although the roots of $S$ and $T$ differ.}}\label{Fig:IsomorphicSubtreeDef}
\end{figure}
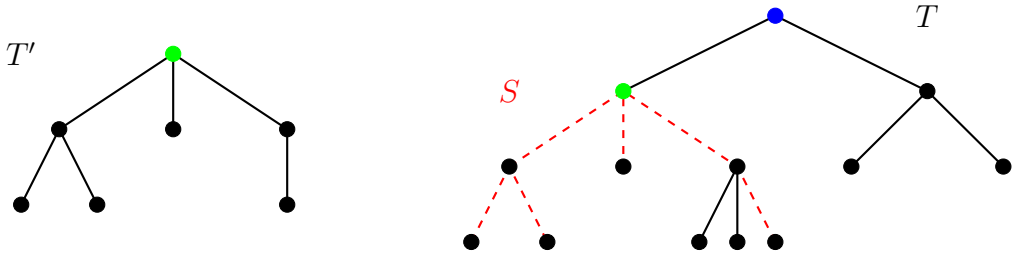

In the language of \cite{AbbBacHanWilZam2015}, the subtree isomorphism we consider is on \textit{ordered} trees, meaning the order of a vertex's children must be preserved by the isomorphism. The fastest known algorithm for this version of the subtree isomorphism problem (on unordered trees) has time complexity $O(n^{\omega+o(1)})$ (as shown in \cite{ShaTsu1999}), where $\omega \leq 2.371$ is the exponent of matrix multiplication, as recently improved by \cite{Wil2024}. We seek to improve upon that bound for vertex-colored arborescences. 

\subsection{Overview}

Despite the typical definition, we permit monochromatic edges in the ``vertex colorings'' considered throughout. To improve readability, we separate most formal definitions and proofs from the motivation and intuition for the study. In Section \ref{Sec:OverarchingIdeas} we motivate the extension of Pr\"{u}fer codes to the setting of vertex-colored arborescences, outline the definition, and describe ways to determine structural relationships between two arborescences based on relationships observed in their VCPCs. In Section \ref{Sec:Application} we apply the concepts presented earlier in Section \ref{Sec:OverarchingIdeas} to the matroid decomposition of the freight dataset considered in \cite{FreightData}, summarizing commonly occurring and rare structure present among those components. In Section \ref{Sec:FormalAlgs} we present formal algorithms and definitions surrounding VCPCs, while Section \ref{Sec:CompClassify} contains many of the results surrounding structural relationships between arborescences discernable from VCPCs. Finally, in Section \ref{Sec:FutureDirections} we provide some ideas for additional work. 

\section{Overarching Ideas}\label{Sec:OverarchingIdeas}

Throughout we assume all graphs are simple and finite unless otherwise specified. Frequently we define $T$ to be an arborescence; recall that an arborescence is a tree in which all edges are directed ``away from'' the root. Given $v \in V(T)$, define $T[v]$ to be the subtree of $T$ containing $v$ and all descendants of $v$. In other words, $T[v]$ is the out-component of $v$ in $T$, i.e., the vertices reachable by a path originating at $v$. Furthermore, the depth of $v \in V(T)$ is the distance between $v$ and the root of $T$. 

We make use of pseudocode to provide presentations of various algorithms. We regularly use arrays and dictionaries. Use $[]$ to denote the empty array. Otherwise, we denote an array as comma-separated values within square brackets. Use $\{\}$ to denote the empty dictionary. Otherwise, we denote a dictionary as comma-separated pairs within curly braces, where pairs are separated by ``:''.

\subsection{Motivation}

While the graph isomorphism problem is quasipolynomial time (\cite{bab}), this problem can be resolved in polynomial time for labeled trees. In this setting, Pr\"{u}fer codes provide a quick check. Given tree $T$ on $n$ vertices with distinct labels zero through $n-1$, the tree is converted to an array of length $n-2$ on the alphabet $\{0,\dotsc,n-1\}$ by identifying the leaf with minimum label, recording the label of its unique neighbor, pruning the leaf (i.e., removing the leaf and its incident edge from the tree), and terminating when only one edge remains. The following figure provides an example, where the Pr\"{u}fer code is updated at each step under the image of the remaining tree. 
\begin{figure}[H]
\centering
\begin{minipage}{0.22\textwidth}
\centering
\begin{tikzpicture}
\draw[thick,black] (0,2) -- (0,1);
\draw[thick,black] (0,0) -- (0,1);
\draw[thick,black] (0,0) -- (1,1);
\draw[thick,black] (0,0) -- (-1,1);
\draw[thick,black] (0,0) -- (-1,-1);
\draw[thick,black] (0,0) -- (1,-1);

\filldraw[black](0,2)circle(0.1);
\filldraw[black](0,1)circle(0.1);
\filldraw[black](0,0)circle(0.1);
\filldraw[black](1,1)circle(0.1);
\filldraw[black](1,-1)circle(0.1);
\filldraw[black](-1,-1)circle(0.1);
\filldraw[black](-1,1)circle(0.1);

\node at (0,-0.35) {4};
\node at (1,-0.65) {2};
\node at (1,0.65) {5};
\node at (-1,-0.65) {0};
\node at (-1,0.65) {6};
\node at (0.35,1) {1};
\node at (0,2.35) {3};

\node at (0,-2) {[]};
\end{tikzpicture}
\end{minipage}
\begin{minipage}{0.22\textwidth}
\centering
\begin{tikzpicture}
\draw[thick,black] (0,2) -- (0,1);
\draw[thick,black] (0,0) -- (0,1);
\draw[thick,black] (0,0) -- (1,1);
\draw[thick,black] (0,0) -- (-1,1);
\draw[thick,black] (0,0) -- (1,-1);

\filldraw[black](0,2)circle(0.1);
\filldraw[black](0,1)circle(0.1);
\filldraw[black](0,0)circle(0.1);
\filldraw[black](1,1)circle(0.1);
\filldraw[black](1,-1)circle(0.1);
\filldraw[black](-1,1)circle(0.1);

\node at (0,-0.35) {4};
\node at (1,-0.65) {2};
\node at (1,0.65) {5};
\node at (-1,0.65) {6};
\node at (0.35,1) {1};
\node at (0,2.35) {3};

\node at (0,-2) {[4]};
\end{tikzpicture}
\end{minipage}
\begin{minipage}{0.22\textwidth}
\centering
\begin{tikzpicture}
\draw[thick,black] (0,2) -- (0,1);
\draw[thick,black] (0,0) -- (0,1);
\draw[thick,black] (0,0) -- (1,1);
\draw[thick,black] (0,0) -- (-1,1);

\filldraw[black](0,2)circle(0.1);
\filldraw[black](0,1)circle(0.1);
\filldraw[black](0,0)circle(0.1);
\filldraw[black](1,1)circle(0.1);
\filldraw[black](-1,1)circle(0.1);

\node at (0,-0.35) {4};
\node at (1,0.65) {5};
\node at (-1,0.65) {6};
\node at (0.35,1) {1};
\node at (0,2.35) {3};

\node at (0,-2) {[4, 4]};
\end{tikzpicture}
\end{minipage}
\begin{minipage}{0.22\textwidth}
\centering
\begin{tikzpicture}
\draw[thick,black] (0,0) -- (0,1);
\draw[thick,black] (0,0) -- (1,1);
\draw[thick,black] (0,0) -- (-1,1);

\filldraw[black](0,1)circle(0.1);
\filldraw[black](0,0)circle(0.1);
\filldraw[black](1,1)circle(0.1);
\filldraw[black](-1,1)circle(0.1);

\node at (0,-0.35) {4};
\node at (1,0.65) {5};
\node at (-1,0.65) {6};
\node at (0.35,1) {1};
\node at (0,2.35) {\textcolor{white}{3}};

\node at (0,-2) {[4, 4, 1]};
\end{tikzpicture}
\end{minipage}
\begin{minipage}{0.22\textwidth}
\centering
\begin{tikzpicture}
\draw[thick,black] (0,0) -- (1,1);
\draw[thick,black] (0,0) -- (-1,1);

\filldraw[black](0,0)circle(0.1);
\filldraw[black](1,1)circle(0.1);
\filldraw[black](-1,1)circle(0.1);

\node at (0,-0.35) {4};
\node at (1,0.65) {5};
\node at (-1,0.65) {6};
\node at (0,2.35) {\textcolor{white}{3}};

\node at (0,-2) {[4, 4, 1, 4]};
\end{tikzpicture}
\end{minipage}
\begin{minipage}{0.22\textwidth}
\centering
\begin{tikzpicture}
\draw[thick,black] (0,0) -- (-1,1);

\filldraw[black](0,0)circle(0.1);
\filldraw[black](-1,1)circle(0.1);

\node at (0,-0.35) {4};
\node at (-1,0.65) {6};
\node at (0,2.35) {\textcolor{white}{3}};

\node at (0,-2) {[4, 4, 1, 4, 4]};
\end{tikzpicture}
\end{minipage}
\caption{Pr\"{u}fer Code example}\label{Fig:Prufer}
\end{figure}
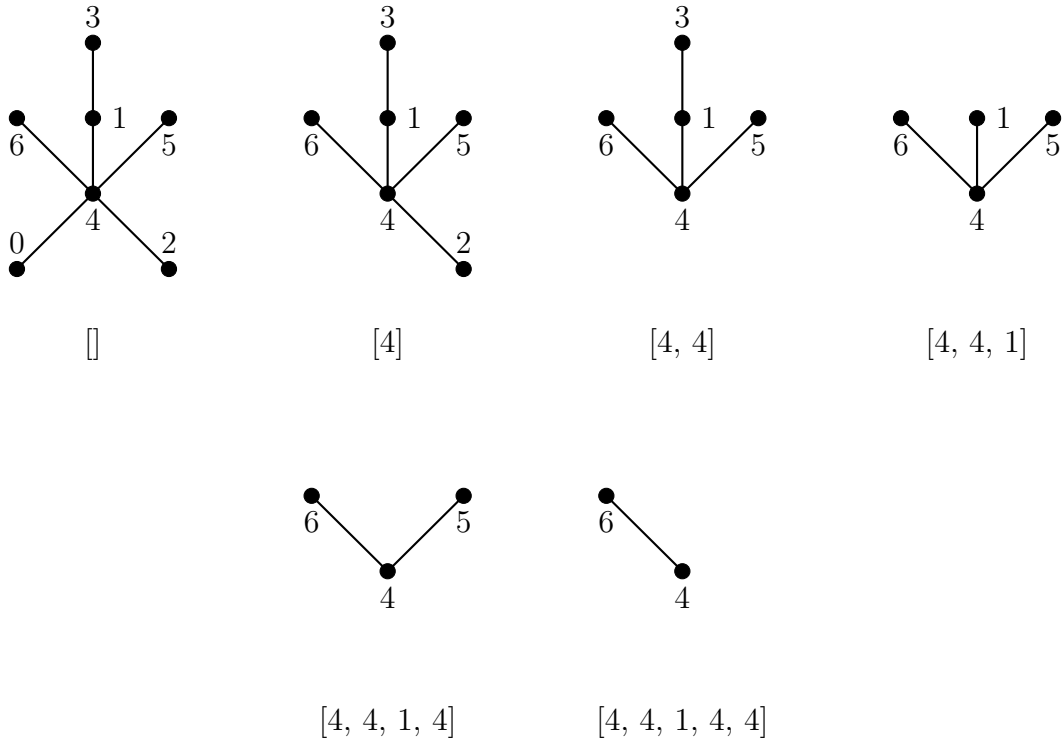

Now we provide a recursive definition for Pr\"{u}fer codes. Suppose $T = (V(T), E(T))$ is a tree on $n$ vertices, and $\phi : V(T) \to\{0,\dotsc,n-1\}$ is the accompanying vertex labeling. 
The optional argument $pc$ is the array which ultimately contains the Pr\"{u}fer code. 

\begin{algorithm}[H]
\caption{\texttt{Pr\"{u}fer}$(T, \phi, pc=[]))$ - Building the Pr\"{u}fer code}
\begin{algorithmic}
\IF{$|V(T)|$ is $2$}
\RETURN pc
\ELSE
\STATE $\ell \gets$ the leaf of $T$ which minimizes $\phi$
\STATE $v \gets$ the unique neighbor of $\ell$ in $T$
\STATE append $\phi(v)$ to $pc$
\STATE remove vertex $\ell$ from T
\RETURN \texttt{Pr\"{u}fer}$(T, \phi, pc)$
\ENDIF
\end{algorithmic}
\end{algorithm}

Turning a Pr\"{u}fer code back into a labeled tree is also straightforward. Given code $pc = [4,4,1,4,4]$, of length $5$, the corresponding tree $T$ must have seven vertices, since the code for a tree on $n$ vertices has length $n-2$. So define $V(T) = \{0,\dotsc, 6\}$. Initialize $E(T) = \emptyset$. Let $L$ be the set given by $L = \{0,2,3,5,6\}$, which contains vertices of $V(T)$ except those contained in $pc$. Loop through the entries of $pc$. At each iteration, add $(pc[0], \min(L))$ to $E(T)$, update $v = pc[0]$, remove the minimum element of $L$, and remove the zeroth element of $pc$. If $v$ does not occur in $pc$ (note that $v$ was the zeroth element of $pc$ but this entry was removed with the update), include it in $L$. Iterate until $pc$ is empty, in which case $|L|=2$. Include the edge connecting the two vertices of $L$ as the final edge in $E(T)$. The following figure illustrates an example of this process, where $pc$ and $L$ are updated at each step under the image of the current forest. 

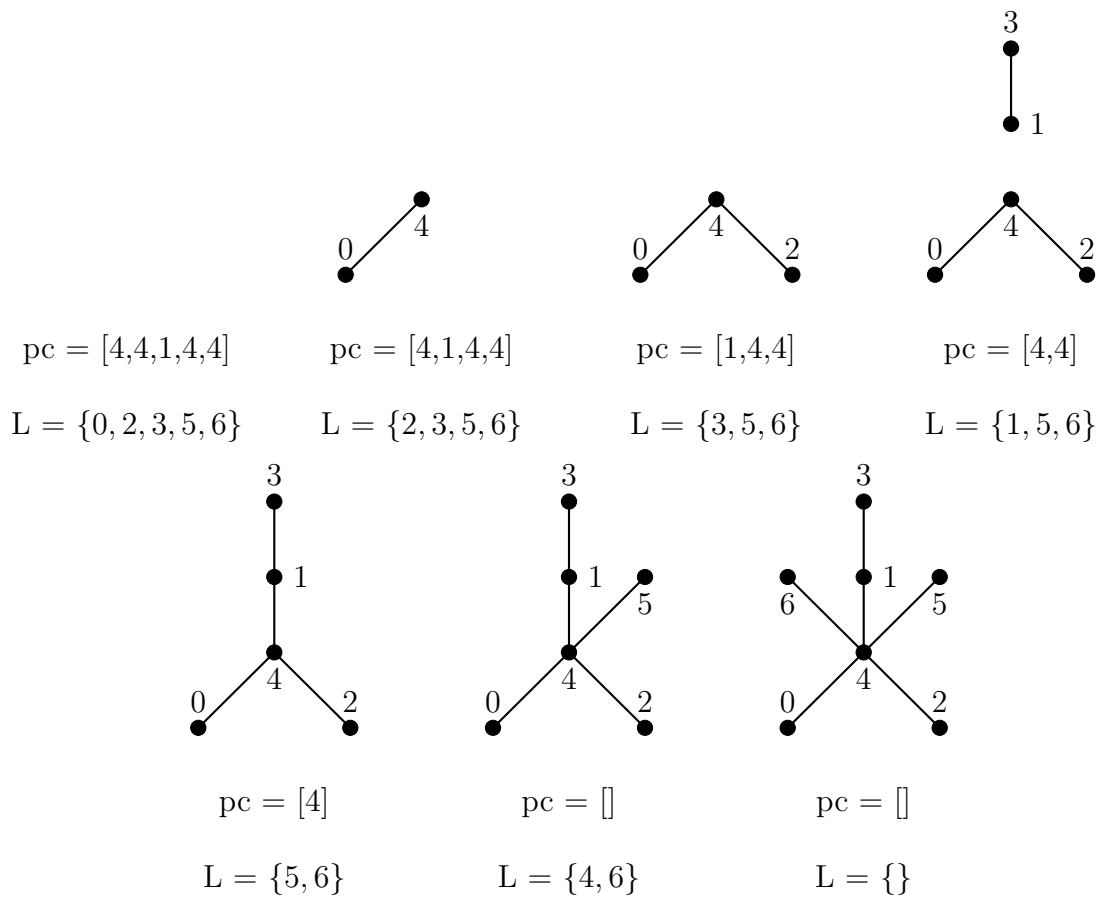
\begin{figure}[H]
\centering
\begin{minipage}{0.22\textwidth}
\centering
\begin{tikzpicture}
\draw[thick,white] (0,2) -- (0,1);
\draw[thick,white] (0,0) -- (0,1);
\draw[thick,white] (0,0) -- (1,1);
\draw[thick,white] (0,0) -- (-1,1);
\draw[thick,white] (0,0) -- (-1,-1);
\draw[thick,white] (0,0) -- (1,-1);

\filldraw[white](0,2)circle(0.1);
\filldraw[white](0,1)circle(0.1);
\filldraw[white](0,0)circle(0.1);
\filldraw[white](1,1)circle(0.1);
\filldraw[white](1,-1)circle(0.1);
\filldraw[white](-1,-1)circle(0.1);
\filldraw[white](-1,1)circle(0.1);

\node at (0,-0.35) {\textcolor{white}{4}};
\node at (1,-0.65) {\textcolor{white}{2}};
\node at (1,0.65) {\textcolor{white}{5}};
\node at (-1,-0.65) {\textcolor{white}{0}};
\node at (-1,0.65) {\textcolor{white}{6}};
\node at (0.35,1) {\textcolor{white}{1}};
\node at (0,2.35) {\textcolor{white}{3}};

\node at (0,-2) {pc = [4,4,1,4,4]};
\node at (0,-3) {L = $\{0,2,3,5,6\}$};
\end{tikzpicture}
\end{minipage}
\begin{minipage}{0.22\textwidth}
\centering
\begin{tikzpicture}
\draw[thick,white] (0,2) -- (0,1);
\draw[thick,white] (0,0) -- (0,1);
\draw[thick,white] (0,0) -- (1,1);
\draw[thick,white] (0,0) -- (-1,1);
\draw[thick,black] (0,0) -- (-1,-1);
\draw[thick,white] (0,0) -- (1,-1);

\filldraw[white](0,2)circle(0.1);
\filldraw[white](0,1)circle(0.1);
\filldraw[black](0,0)circle(0.1);
\filldraw[white](1,1)circle(0.1);
\filldraw[white](1,-1)circle(0.1);
\filldraw[black](-1,-1)circle(0.1);
\filldraw[white](-1,1)circle(0.1);

\node at (0,-0.35) {\textcolor{black}{4}};
\node at (1,-0.65) {\textcolor{white}{2}};
\node at (1,0.65) {\textcolor{white}{5}};
\node at (-1,-0.65) {\textcolor{black}{0}};
\node at (-1,0.65) {\textcolor{white}{6}};
\node at (0.35,1) {\textcolor{white}{1}};
\node at (0,2.35) {\textcolor{white}{3}};

\node at (0,-2) {pc = [4,1,4,4]};
\node at (0,-3) {L = $\{2,3,5,6\}$};
\end{tikzpicture}
\end{minipage}
\begin{minipage}{0.22\textwidth}
\centering
\begin{tikzpicture}
\draw[thick,white] (0,2) -- (0,1);
\draw[thick,white] (0,0) -- (0,1);
\draw[thick,white] (0,0) -- (1,1);
\draw[thick,white] (0,0) -- (-1,1);
\draw[thick,black] (0,0) -- (-1,-1);
\draw[thick,black] (0,0) -- (1,-1);

\filldraw[white](0,2)circle(0.1);
\filldraw[white](0,1)circle(0.1);
\filldraw[black](0,0)circle(0.1);
\filldraw[white](1,1)circle(0.1);
\filldraw[black](1,-1)circle(0.1);
\filldraw[black](-1,-1)circle(0.1);
\filldraw[white](-1,1)circle(0.1);

\node at (0,-0.35) {\textcolor{black}{4}};
\node at (1,-0.65) {\textcolor{black}{2}};
\node at (1,0.65) {\textcolor{white}{5}};
\node at (-1,-0.65) {\textcolor{black}{0}};
\node at (-1,0.65) {\textcolor{white}{6}};
\node at (0.35,1) {\textcolor{white}{1}};
\node at (0,2.35) {\textcolor{white}{3}};

\node at (0,-2) {pc = [1,4,4]};
\node at (0,-3) {L = $\{3,5,6\}$};
\end{tikzpicture}
\end{minipage}
\begin{minipage}{0.22\textwidth}
\centering
\begin{tikzpicture}
\draw[thick,black] (0,2) -- (0,1);
\draw[thick,white] (0,0) -- (0,1);
\draw[thick,white] (0,0) -- (1,1);
\draw[thick,white] (0,0) -- (-1,1);
\draw[thick,black] (0,0) -- (-1,-1);
\draw[thick,black] (0,0) -- (1,-1);

\filldraw[black](0,2)circle(0.1);
\filldraw[black](0,1)circle(0.1);
\filldraw[black](0,0)circle(0.1);
\filldraw[white](1,1)circle(0.1);
\filldraw[black](1,-1)circle(0.1);
\filldraw[black](-1,-1)circle(0.1);
\filldraw[white](-1,1)circle(0.1);

\node at (0,-0.35) {\textcolor{black}{4}};
\node at (1,-0.65) {\textcolor{black}{2}};
\node at (1,0.65) {\textcolor{white}{5}};
\node at (-1,-0.65) {\textcolor{black}{0}};
\node at (-1,0.65) {\textcolor{white}{6}};
\node at (0.35,1) {\textcolor{black}{1}};
\node at (0,2.35) {\textcolor{black}{3}};

\node at (0,-2) {pc = [4,4]};
\node at (0,-3) {L = $\{1,5,6\}$};
\end{tikzpicture}
\end{minipage}
\begin{minipage}{0.22\textwidth}
\centering
\begin{tikzpicture}
\draw[thick,black] (0,2) -- (0,1);
\draw[thick,black] (0,0) -- (0,1);
\draw[thick,white] (0,0) -- (1,1);
\draw[thick,white] (0,0) -- (-1,1);
\draw[thick,black] (0,0) -- (-1,-1);
\draw[thick,black] (0,0) -- (1,-1);

\filldraw[black](0,2)circle(0.1);
\filldraw[black](0,1)circle(0.1);
\filldraw[black](0,0)circle(0.1);
\filldraw[white](1,1)circle(0.1);
\filldraw[black](1,-1)circle(0.1);
\filldraw[black](-1,-1)circle(0.1);
\filldraw[white](-1,1)circle(0.1);

\node at (0,-0.35) {\textcolor{black}{4}};
\node at (1,-0.65) {\textcolor{black}{2}};
\node at (1,0.65) {\textcolor{white}{5}};
\node at (-1,-0.65) {\textcolor{black}{0}};
\node at (-1,0.65) {\textcolor{white}{6}};
\node at (0.35,1) {\textcolor{black}{1}};
\node at (0,2.35) {\textcolor{black}{3}};

\node at (0,-2) {pc = [4]};
\node at (0,-3) {L = $\{5,6\}$};
\end{tikzpicture}
\end{minipage}
\begin{minipage}{0.22\textwidth}
\centering
\begin{tikzpicture}
\draw[thick,black] (0,2) -- (0,1);
\draw[thick,black] (0,0) -- (0,1);
\draw[thick,black] (0,0) -- (1,1);
\draw[thick,white] (0,0) -- (-1,1);
\draw[thick,black] (0,0) -- (-1,-1);
\draw[thick,black] (0,0) -- (1,-1);

\filldraw[black](0,2)circle(0.1);
\filldraw[black](0,1)circle(0.1);
\filldraw[black](0,0)circle(0.1);
\filldraw[black](1,1)circle(0.1);
\filldraw[black](1,-1)circle(0.1);
\filldraw[black](-1,-1)circle(0.1);
\filldraw[white](-1,1)circle(0.1);

\node at (0,-0.35) {\textcolor{black}{4}};
\node at (1,-0.65) {\textcolor{black}{2}};
\node at (1,0.65) {\textcolor{black}{5}};
\node at (-1,-0.65) {\textcolor{black}{0}};
\node at (-1,0.65) {\textcolor{white}{6}};
\node at (0.35,1) {\textcolor{black}{1}};
\node at (0,2.35) {\textcolor{black}{3}};

\node at (0,-2) {pc = []};
\node at (0,-3) {L = $\{4,6\}$};
\end{tikzpicture}
\end{minipage}
\begin{minipage}{0.22\textwidth}
\centering
\begin{tikzpicture}
\draw[thick,black] (0,2) -- (0,1);
\draw[thick,black] (0,0) -- (0,1);
\draw[thick,black] (0,0) -- (1,1);
\draw[thick,black] (0,0) -- (-1,1);
\draw[thick,black] (0,0) -- (-1,-1);
\draw[thick,black] (0,0) -- (1,-1);

\filldraw[black](0,2)circle(0.1);
\filldraw[black](0,1)circle(0.1);
\filldraw[black](0,0)circle(0.1);
\filldraw[black](1,1)circle(0.1);
\filldraw[black](1,-1)circle(0.1);
\filldraw[black](-1,-1)circle(0.1);
\filldraw[black](-1,1)circle(0.1);

\node at (0,-0.35) {\textcolor{black}{4}};
\node at (1,-0.65) {\textcolor{black}{2}};
\node at (1,0.65) {\textcolor{black}{5}};
\node at (-1,-0.65) {\textcolor{black}{0}};
\node at (-1,0.65) {\textcolor{black}{6}};
\node at (0.35,1) {\textcolor{black}{1}};
\node at (0,2.35) {\textcolor{black}{3}};

\node at (0,-2) {pc = []};
\node at (0,-3) {L = \{\}};
\end{tikzpicture}
\end{minipage}
\caption{Building a tree from its Pr\"{u}fer code}\label{Fig:InversePrufer}
\end{figure}

In this way, Pr\"{u}fer codes provide a bijection from labeled trees on $n$ vertices to the set of tuples $\{0,\dotsc, n-1\}^{n-2}$. Since a bijection is by definition an injection, two labeled trees $T$ and $T'$ are isomorphic as labeled trees if and only if they have the same Pr\"{u}fer codes. We aim to extend this concept to the setting of unlabeled, vertex-colored arborescences (i.e., a directed tree where all edges are directed away from some ``root'') so that two vertex-colored arborescences $T$ and $T'$ are isomorphic (as vertex-colored directed trees) if and only if their VCPCs are identical. We have three problems to address, namely (1) how can a vertex labeling be applied to a vertex-colored arborescence to allow for a typical Pr\"{u}fer encoding? (2) How can the Pr\"{u}fer code be adapted to contain vertex color information? (3) How can the Pr\"{u}fer code be adapted to a directed tree?

\subsection{VCPCs}

These three problems are listed in descending order of complexity. Thus, we handle the questions in reverse order. To accommodate directed edges, define a leaf of arborescence $T$ to be a vertex with out-degree zero and in-degree at most one. In other words, these are exactly the leaves of the undirected tree underlying $T$. Thus, this definition ensures that all (nonempty) arborescences have at least one leaf. 

Next, we address the encoding of vertex colors. Assume a vertex labeling has already been determined. Then, similar to the work of \cite{ChoKimSeoShi2004} for edge-colored Pr\"{u}fer codes, vertex colors are recorded in the Pr\"{u}fer code by adding a second row containing the color of the vertex pruned. To store the color of every vertex, we extend the code to have length $n$. The first $n-2$ columns will be built by recording the typical Pr\"{u}fer code in the first row and the color of the vertex pruned in the second row. Once only one edge remains, perform another pruning step, populating the $(n-1)$-st column of the code. Then one vertex will remain, so we remove the last vertex, insert $\emptyset$ in the first row of the Pr\"{u}fer code and the color of the last vertex in the second row. We illustrate this in the following example. The VCPC is updated underneath the remaining tree after each iteration. 

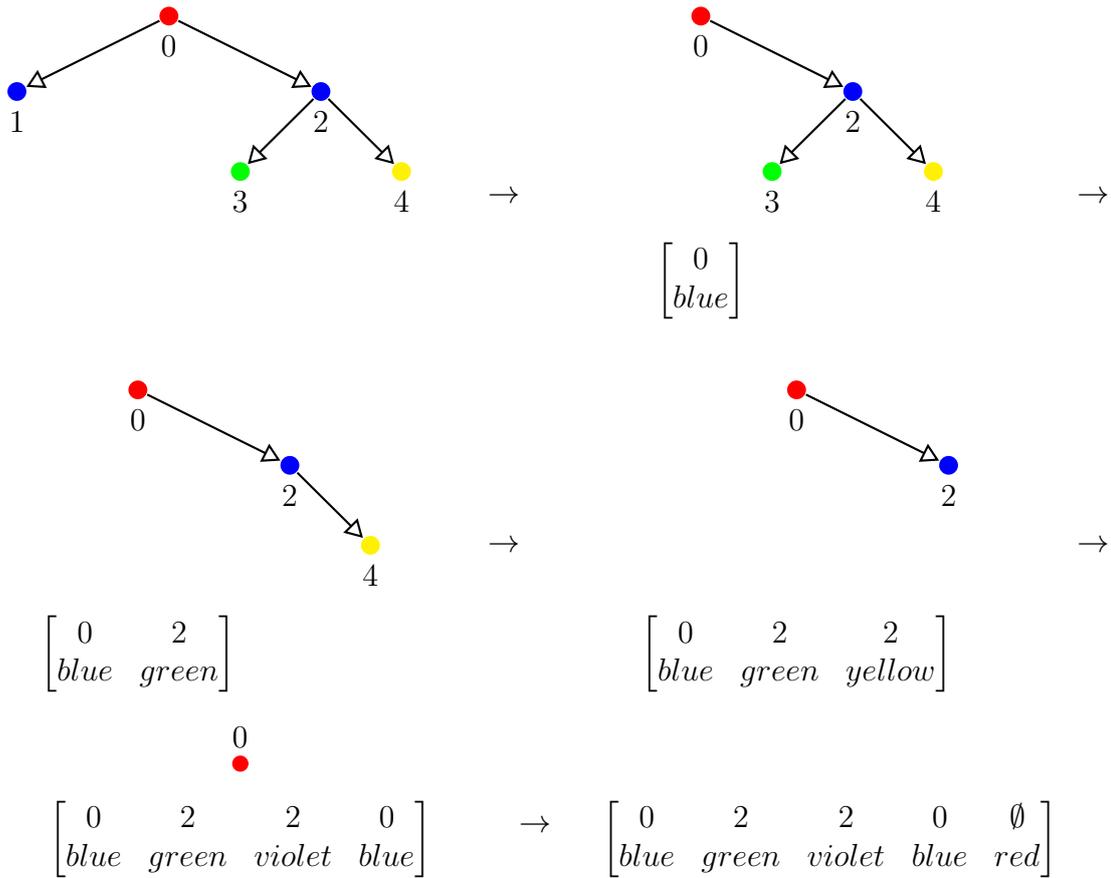
\begin{figure}[H]
\centering
\begin{minipage}{0.4\textwidth}
\centering
\begin{tikzpicture}[thick, -open triangle 60, main/.style = {node distance = 15mm, fill, circle, inner sep = 2.5pt}, label/.style = {node distance = 4mm}]
\node[main,red] (1) {};
\node[main,blue] (2) at (-2,-1) {};
\node[main,blue] (3) at (2,-1) {};
\node[main,green] (6) [below left of=3] {};
\node[main,yellow] (7) [below right of=3] {};

\node at (0,-3.5) {\textcolor{white}{$\begin{bmatrix} 0 \\ blue \end{bmatrix}$}};
\node[label] [below of=1] {0};
\node[label] [below of=2] {1};
\node[label] [below of=3] {2};
\node[label] [below of=6] {3};
\node[label] [below of=7] {4};

\draw (1) -- (2);
\draw (1) -- (3);
\draw (3) -- (6);
\draw (3) -- (7);
\end{tikzpicture}
\end{minipage}
\begin{minipage}{0.04\textwidth}
$$ \rightarrow$$
\end{minipage}
\begin{minipage}{0.4\textwidth}
\centering
\begin{tikzpicture}[thick, -open triangle 60, main/.style = {node distance = 15mm, fill, circle, inner sep = 2.5pt}, label/.style = {node distance = 4mm}]
\node[main,red] (1) {};
\node[main,blue] (3) at (2,-1) {};
\node[main,green] (6) [below left of=3] {};
\node[main,yellow] (7) [below right of=3] {};

\node at (0,-3.5) {$\begin{bmatrix} 0 \\ blue \end{bmatrix}$};
\node[label] [below of=1] {0};
\node[label] [below of=3] {2};
\node[label] [below of=6] {3};
\node[label] [below of=7] {4};

\draw (1) -- (3);
\draw (3) -- (6);
\draw (3) -- (7);
\end{tikzpicture}
\end{minipage}
\begin{minipage}{0.04\textwidth}
$$ \rightarrow$$
\end{minipage}
\begin{minipage}{0.4\textwidth}
\centering
\begin{tikzpicture}[thick, -open triangle 60, main/.style = {node distance = 15mm, fill, circle, inner sep = 2.5pt}, label/.style = {node distance = 4mm}]
\node[main,red] (1) {};
\node[main,blue] (3) at (2,-1) {};
\node[main,yellow] (7) [below right of=3] {};

\node at (0,-3.5) {$\begin{bmatrix} 0 & 2\\ blue & green \end{bmatrix}$};
\node[label] [below of=1] {0};
\node[label] [below of=3] {2};
\node[label] [below of=7] {4};

\node at (0,0.5) {\textcolor{white}{0}};

\draw (1) -- (3);
\draw (3) -- (7);
\end{tikzpicture}
\end{minipage}
\begin{minipage}{0.04\textwidth}
$$ \rightarrow$$
\end{minipage}
\begin{minipage}{0.4\textwidth}
\centering
\begin{tikzpicture}[thick, -open triangle 60, main/.style = {node distance = 15mm, fill, circle, inner sep = 2.5pt}, label/.style = {node distance = 4mm}]
\node[main,red] (1) {};
\node[main,blue] (3) at (2,-1) {};

\node at (0,-3.5) {$\begin{bmatrix} 0 & 2 & 2\\ blue & green & yellow \end{bmatrix}$};
\node[label] [below of=1] {0};
\node[label] [below of=3] {2};
\node at (0,0.5) {\textcolor{white}{0}};

\draw (1) -- (3);
\end{tikzpicture}
\end{minipage}
\begin{minipage}{0.04\textwidth}
$$ \rightarrow$$
\end{minipage}
\begin{minipage}{0.4\textwidth}
\centering
\begin{tikzpicture}
\filldraw[red](0,0)circle(0.1);

\node at (0,0.5) {\textcolor{white}{0}};
\node at (0,0.35) {0};

\node at (0,-1) {$\begin{bmatrix} 0 & 2 & 2 & 0\\ blue & green & violet & blue \end{bmatrix}$};
\end{tikzpicture}
\end{minipage}
\begin{minipage}{0.04\textwidth}
$$ \rightarrow$$
\end{minipage}
\begin{minipage}{0.4\textwidth}
\centering
\begin{tikzpicture}
\filldraw[white](0,0)circle(0.1);

\node at (0,0.5) {\textcolor{white}{0}};

\node at (0,-1) {$\begin{bmatrix} 0 & 2 & 2 & 0 &\emptyset\\ blue & green & violet & blue & red\end{bmatrix}$};
\end{tikzpicture}
\end{minipage}
\caption{Building the VCPC. Recall that a vertex is eligible for pruning only if it has out-degree zero and in-degree at most one. }\label{Fig:PruferExample}
\end{figure}

It only remains to address the first question by defining the vertex labeling. Notice that in the tree from the previous example, changing the labeling slightly will change the resulting Pr\"{u}fer code. Thus, used directly, the VCPC does not provide a valid canonicalization for vertex-colored arborescences. Here are a couple of examples. 
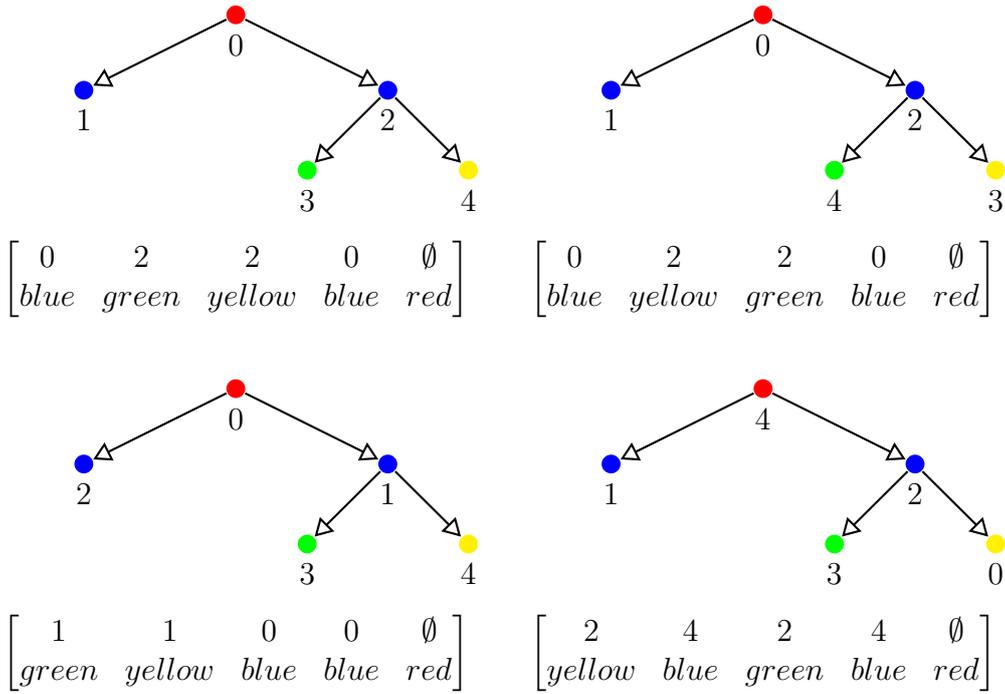
\begin{figure}[H]
\centering
\begin{minipage}{0.4\textwidth}
\centering
\begin{tikzpicture}[thick, -open triangle 60, main/.style = {node distance = 15mm, fill, circle, inner sep = 2.5pt}, label/.style = {node distance = 4mm}]
\node[main,red] (1) {};
\node[main,blue] (2) at (-2,-1) {};
\node[main,blue] (3) at (2,-1) {};
\node[main,green] (6) [below left of=3] {};
\node[main,yellow] (7) [below right of=3] {};

\node at (0,-3.5) {$\begin{bmatrix} 0 & 2 & 2 & 0 &\emptyset\\ blue & green & yellow & blue & red\end{bmatrix}$};
\node[label] [below of=1] {0};
\node[label] [below of=2] {1};
\node[label] [below of=3] {2};
\node[label] [below of=6] {3};
\node[label] [below of=7] {4};

\draw (1) -- (2);
\draw (1) -- (3);
\draw (3) -- (6);
\draw (3) -- (7);
\end{tikzpicture}
\end{minipage}
\begin{minipage}{0.4\textwidth}
\centering
\begin{tikzpicture}[thick, -open triangle 60, main/.style = {node distance = 15mm, fill, circle, inner sep = 2.5pt}, label/.style = {node distance = 4mm}]
\node[main,red] (1) {};
\node[main,blue] (2) at (-2,-1) {};
\node[main,blue] (3) at (2,-1) {};
\node[main,green] (6) [below left of=3] {};
\node[main,yellow] (7) [below right of=3] {};

\node at (0,-3.5) {$\begin{bmatrix} 0 & 2 & 2 & 0 &\emptyset\\ blue & yellow & green & blue & red\end{bmatrix}$};
\node[label] [below of=1] {0};
\node[label] [below of=2] {1};
\node[label] [below of=3] {2};
\node[label] [below of=6] {4};
\node[label] [below of=7] {3};

\draw (1) -- (2);
\draw (1) -- (3);
\draw (3) -- (6);
\draw (3) -- (7);
\end{tikzpicture}
\end{minipage}
\begin{minipage}{0.4\textwidth}
\centering
\begin{tikzpicture}[thick, -open triangle 60, main/.style = {node distance = 15mm, fill, circle, inner sep = 2.5pt}, label/.style = {node distance = 4mm}]
\node[main,red] (1) {};
\node[main,blue] (2) at (-2,-1) {};
\node[main,blue] (3) at (2,-1) {};
\node[main,green] (6) [below left of=3] {};
\node[main,yellow] (7) [below right of=3] {};

\node at (0,-3.5) {$\begin{bmatrix} 1 & 1 & 0 & 0 &\emptyset\\ green & yellow & blue & blue &  red\end{bmatrix}$};
\node[label] [below of=1] {0};
\node[label] [below of=2] {2};
\node[label] [below of=3] {1};
\node[label] [below of=6] {3};
\node[label] [below of=7] {4};

\node at (0,0.5) {\textcolor{white}{0}};

\draw (1) -- (2);
\draw (1) -- (3);
\draw (3) -- (6);
\draw (3) -- (7);
\end{tikzpicture}
\end{minipage}
\begin{minipage}{0.4\textwidth}
\centering
\begin{tikzpicture}[thick, -open triangle 60, main/.style = {node distance = 15mm, fill, circle, inner sep = 2.5pt}, label/.style = {node distance = 4mm}]
\node[main,red] (1) {};
\node[main,blue] (2) at (-2,-1) {};
\node[main,blue] (3) at (2,-1) {};
\node[main,green] (6) [below left of=3] {};
\node[main,yellow] (7) [below right of=3] {};

\node at (0,-3.5) {$\begin{bmatrix} 2 & 4 & 2 & 4 &\emptyset\\ yellow & blue & green & blue & red\end{bmatrix}$};
\node[label] [below of=1] {4};
\node[label] [below of=2] {1};
\node[label] [below of=3] {2};
\node[label] [below of=6] {3};
\node[label] [below of=7] {0};

\node at (0,0.5) {\textcolor{white}{0}};

\draw (1) -- (2);
\draw (1) -- (3);
\draw (3) -- (6);
\draw (3) -- (7);
\end{tikzpicture}
\end{minipage}
\caption{VCPC for same tree with different labeling}\label{Fig:PruferAltExample}
\end{figure}

Given an arborescence $T$, let $T[v]$ denote the subtree of $T$ containing $v$ and all descendants of $v$. Let $c : V(T) \to \mathbb{N}$ be the coloring function on $V(T)$ (recall that we do not require $c$ be a proper coloring, i.e., we permit the existence of monochromatic edges). Eventually we assign the vertex-labeling of $T$ as the order of traversal in a depth first search (starting from the root), where we let $u <_T v$ denote that $u$ is traversed before $v$. The children of $v \in V(T)$ are traversed in the order given by the ordering of their colors, however, there could be ties. Within a color, we break ties among sibling vertices $v$ by the isomorphism-class of their subtrees $T[v]$. We make this tiebreaking more precise in Algorithm \ref{Alg:GetMinVert}. Notice though that there may still be a tie if $c(v_1) = c(v_2)$ and $T[v_1]$ and $T[v_2]$ are isomorphic as vertex-colored arborescences. However, we make no additional attempt to break this tie and claim this presents no issue. Consider the following example where we abbreviate colors blue, green, red, and yellow by their first letters. 
\begin{figure}[H]
\centering
\begin{minipage}{0.45\textwidth}
\centering
\begin{tikzpicture}[thick, -open triangle 60, main/.style = {node distance = 15mm, fill, circle, inner sep = 2.5pt}, label/.style = {node distance = 4mm}]
\node[main,red] (1) {};
\node[main,blue] (2) at (-2,-1) {};
\node[main,blue] (3) at (2,-1) {};
\node[main,green] (4) [below left of=2] {};
\node[main,yellow] (5) [below right of=2] {};
\node[main,green] (6) [below left of=3] {};
\node[main,yellow] (7) [below right of=3] {};

\node at (0,-3.5) {$\begin{bmatrix} 1 & 1 & 0 & 4 & 4 & 0 &\emptyset\\ g& y& b & g & y & b & r\end{bmatrix}$};
\node[label] [below of=1] {0};
\node[label] [below of=2] {1};
\node[label] [below of=3] {4};
\node[label] [below of=4] {2};
\node[label] [below of=5] {3};
\node[label] [below of=6] {5};
\node[label] [below of=7] {6};

\draw (1) -- (2);
\draw (1) -- (3);
\draw (2) -- (4);
\draw (2) -- (5);
\draw (3) -- (6);
\draw (3) -- (7);
\end{tikzpicture}
\end{minipage}
\begin{minipage}{0.45\textwidth}
\centering
\begin{tikzpicture}[thick, -open triangle 60, main/.style = {node distance = 15mm, fill, circle, inner sep = 2.5pt}, label/.style = {node distance = 4mm}]
\node[main,red] (1) {};
\node[main,blue] (2) at (-2,-1) {};
\node[main,blue] (3) at (2,-1) {};
\node[main,green] (4) [below left of=2] {};
\node[main,yellow] (5) [below right of=2] {};
\node[main,green] (6) [below left of=3] {};
\node[main,yellow] (7) [below right of=3] {};

\node at (0,-3.5) {$\begin{bmatrix} 1 & 1 & 0 & 4 & 4 & 0 &\emptyset\\ g& y& b & g & y & b & r\end{bmatrix}$};
\node[label] [below of=1] {0};
\node[label] [below of=2] {4};
\node[label] [below of=3] {1};
\node[label] [below of=4] {5};
\node[label] [below of=5] {6};
\node[label] [below of=6] {2};
\node[label] [below of=7] {3};

\draw (1) -- (2);
\draw (1) -- (3);
\draw (2) -- (4);
\draw (2) -- (5);
\draw (3) -- (6);
\draw (3) -- (7);
\end{tikzpicture}
\end{minipage}
\caption{VCPC for a tree with two ``automorphic'' labelings}\label{Fig:PruferAltAutoExample}
\end{figure}
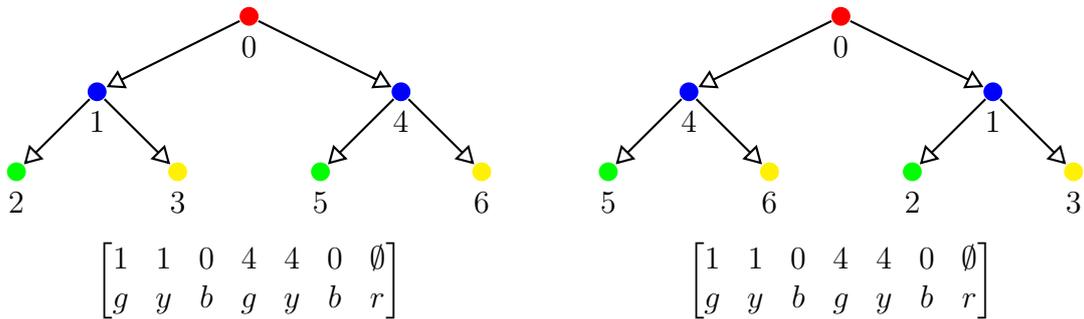

If $T$ and $T'$ are the trees on the left and right of Figure \ref{Fig:PruferAltAutoExample} respectively and $\psi: V(T) \to V(T')$ is the map which preserves the vertex labeling, then $\psi$ is a vertex-colored isomorphism between $T$ and $T'$. In other words, $\psi$ is really a vertex-colored automorphism of $T$, and so both labelings give rise to the same Pr\"{u}fer code. So, we need not break all ties deterministically, only the ones that change the resulting Pr\"{u}fer code. The remaining ties can be broken arbitrarily. In this way, two vertex-colored arborescences are isomorphic if and only if their VCPCs are the same, which we make more precise in Theorem \ref{Prop:VertexColoredPruferIsomorphism}. 

\subsection{Subtree Isomorphism}\label{Sec:SubtreeIso}

It is reasonable to resolve the vertex colored arborescence subisomorphism problem, i.e., the problem of determining if one vertex-colored arborescence is a subtree of another, with VCPCs. We present some examples in the next figure. The graphs appearing in this figure are arranged in rows, and each row forms a separate example. The larger arborescence $T$ appears on the right with a smaller arborescence $T'$ on the left. In the first two examples, $T'$ is isomorphic to a subtree of $T$, whereas this is not the case in the last example. The VCPC is provided for both (again abbreviating colors with their first letter), and the columns of the code for $T$ corresponding to the subtree $T'$ are boxed in blue when appropriate. 

\begin{figure}[H]
\centering
\begin{minipage}{0.45\textwidth}
\centering
\begin{tikzpicture}[thick, -open triangle 60, main/.style = {node distance = 15mm, fill, circle, inner sep = 2.5pt}, label/.style = {node distance = 4mm}]
\node[main,red] (1) {};
\node[main,blue] (2) at (-2,-1) {};
\node[main,blue] (3) at (2,-1) {};
\node[main,green] (6) [below left of=3] {};
\node[main,yellow] (7) [below right of=3] {};

\node at (0,-4) {$\begin{bmatrix} 0 & 2 & 2& 0 & \emptyset \\ b & g & y & b & r \end{bmatrix}$};
\node at (0,-5) {};

\node[label] at (-2,0) {$T'$};

\draw (1) -- (2);
\draw (1) -- (3);
\draw (3) -- (6);
\draw (3) -- (7);
\end{tikzpicture}
\end{minipage}
\begin{minipage}{0.05\textwidth}
$$\rightarrow$$
\end{minipage}
\begin{minipage}{0.45\textwidth}
\centering
\begin{tikzpicture}[thick, -open triangle 60, main/.style = {node distance = 15mm, fill, circle, inner sep = 2.5pt}, label/.style = {node distance = 4mm}]
\node[main,red] (1) {};
\node[main,blue] (2) at (-2,-1) {};
\node[main,blue] (3) at (2,-1) {};
\node[main,green] (4) [below left of=2] {};
\node[main,red] (5) [below right of=2] {};
\node[main,green] (6) [below left of=3] {};
\node[main,yellow] (7) at (2,-2.06066) {};
\node[main,yellow] (8) [below right of=3] {};
\node[main,blue] (9) [below left of=6] {};
\node[main,yellow] (10) [below right of=6] {};

\node at (0,-4) {$\begin{bmatrix} 1&1&0&5&5&4&4&4&0&\emptyset \\ g&r&b&b&y&g&y&y&b&r \end{bmatrix}$};
\node at (0,-5) {};

\node[label] at (-2,0) {$T$};

\draw[blue,dashed] (1) -- (2);
\draw[blue,dashed] (1) -- (3);
\draw (2) -- (4);
\draw (2) -- (5);
\draw[blue,dashed] (3) -- (6);
\draw (3) -- (7);
\draw[blue,dashed] (3) -- (8);
\draw (6) -- (9);
\draw (6) -- (10);
\end{tikzpicture}
\end{minipage}

$T'$ is a subarborescence of $T$ given by the blue (dashed) edges. 

\begin{minipage}{0.45\textwidth}
\centering
\begin{tikzpicture}[thick, -open triangle 60, main/.style = {node distance = 15mm, fill, circle, inner sep = 2.5pt}, label/.style = {node distance = 4mm}]
\node[main,red] (1) {};
\node[main,blue] (2) at (-2,-1) {};
\node[main,blue] (3) at (2,-1) {};
\node[main,green] (6) [below left of=3] {};
\node[main,yellow] (7) [below right of=3] {};

\node at (0,-3) {$\begin{bmatrix} 0 & 2 & 2& 0 & \emptyset \\ b & g & y & b & r \end{bmatrix}$};
\node at (0,-4) {};
\node at (0,1.5) {};

\node[label] at (-2,0) {$T'$};

\draw (1) -- (2);
\draw (1) -- (3);
\draw (3) -- (6);
\draw (3) -- (7);
\end{tikzpicture}
\end{minipage}
\begin{minipage}{0.05\textwidth}
$$\rightarrow$$
\end{minipage}
\begin{minipage}{0.45\textwidth}
\centering
\begin{tikzpicture}[thick, -open triangle 60, main/.style = {node distance = 15mm, fill, circle, inner sep = 2.5pt}, label/.style = {node distance = 4mm}]
\node[main,red] (1) {};
\node[main,blue] (2) at (-2,-1) {};
\node[main,blue] (3) at (2,-1) {};
\node[main,green] (4) [below left of=2] {};
\node[main,red] (5) [below right of=2] {};
\node[main,green] (6) [below left of=3] {};
\node[main,yellow] (7) at (2,-2.06066) {};
\node[main,yellow] (8) [below right of=3] {};

\node[main,green] (9) at (-2,1) {};
\node[main,red] (10) at (-2,0) {};

\node at (0,-3) {$\begin{bmatrix} 0&3&3&2&6&6&6&2&0&\emptyset \\ r&g&r&b&g&y&y&b&r&g \end{bmatrix}$};
\node at (0,-4) {};

\node[label] at (0,1) {$T$};
\node at (0,1.5) {};

\draw[blue,dashed] (1) -- (2);
\draw[blue,dashed] (1) -- (3);
\draw (2) -- (4);
\draw (2) -- (5);
\draw[blue,dashed] (3) -- (6);
\draw (3) -- (7);
\draw[blue,dashed] (3) -- (8);
\draw (9) -- (1);
\draw (9) -- (10);
\end{tikzpicture}
\end{minipage}

$T'$ is a subarborescence of $T$ given by the blue (dashed) edges.

\begin{minipage}{0.45\textwidth}
\centering
\begin{tikzpicture}[thick, -open triangle 60, main/.style = {node distance = 15mm, fill, circle, inner sep = 2.5pt}, label/.style = {node distance = 4mm}]
\node[main,red] (1) {};
\node[main,blue] (2) at (-2,-1) {};
\node[main,blue] (3) at (0,-1) {};
\node[main,green] (4) at (2,-1) {};
\node[main,green] (6) [below left of=3] {};
\node[main,yellow] (7) [below right of=3] {};

\node at (0,-3) {$\begin{bmatrix} 0 & 2 & 2& 0 & 0 & \emptyset  \\ b & g & y & b & g & r \end{bmatrix}$};
\node at (0,-4) {};
\node at (0,1.5) {};

\node[label] at (-2,0) {$T'$};

\draw (1) -- (2);
\draw (1) -- (3);
\draw (1) -- (4);
\draw (3) -- (6);
\draw (3) -- (7);
\end{tikzpicture}
\end{minipage}
\begin{minipage}{0.05\textwidth}
$$\rightarrow$$
\end{minipage}
\begin{minipage}{0.45\textwidth}
\centering
\begin{tikzpicture}[thick, -open triangle 60, main/.style = {node distance = 15mm, fill, circle, inner sep = 2.5pt}, label/.style = {node distance = 4mm}]
\node[main,red] (1) {};
\node[main,blue] (2) at (-2,-1) {};
\node[main,blue] (3) at (2,-1) {};
\node[main,green] (4) [below left of=2] {};
\node[main,red] (5) [below right of=2] {};
\node[main,green] (6) [below left of=3] {};
\node[main,yellow] (7) at (2,-2.06066) {};
\node[main,yellow] (8) [below right of=3] {};

\node[main,green] (9) at (-2,1) {};
\node[main,red] (10) at (-2,0) {};

\node at (0,-3) {$\begin{bmatrix} 0&3&3&2&6&6&6&2&0&\emptyset \\ r&g&r&b&g&y&y&b&r&g \end{bmatrix}$};
\node at (0,-4) {};

\node[label] at (0,1) {$T$};
\node at (0,1.5) {};

\draw (1) -- (2);
\draw (1) -- (3);
\draw (2) -- (4);
\draw (2) -- (5);
\draw (3) -- (6);
\draw (3) -- (7);
\draw (3) -- (8);
\draw (9) -- (1);
\draw (9) -- (10);
\end{tikzpicture}
\end{minipage}

$T'$ is not a subarborescence of $T$ (although the undirected tree underlying $T'$ is a subtree of the undirected tree underlying $T$). 
\caption{Vertex-colored subarborescence examples with accompanying VCPCs.}\label{Fig:SubtreeExamples}
\end{figure}

In Theorem \ref{Thm:SubtreeIsomorphism} we show arborescence subisomorphism can be seen through three simultaneous properties in VCPCs. The first is obvious, namely that colors must agree. The other two ensure the graph structures are identical. 

Let $<_T$ and $<_{T'}$ be the depth-first search orders introduced above Figure \ref{Fig:PruferAltAutoExample} for $T$ and $T'$ respectively (a more formal description of this ordering is provided in Section \ref{Sec:FormalAlgs}). Furthermore, let $P$ and $P'$ be the corresponding VCPCs. If $\psi : V(T') \to V(T)$ is an isomorphism between $T'$ and a subtree of $T$ and $u', v' \in V(T')$, then $u' <_{T'} v'$ if and only if $\psi(u') <_T \psi(v')$. This has a convenient consequence on the relationship between the first rows of $P$ and $P'$, namely that the first row of $P'$ must have the same ``shape'' as the first row of the columns of $P$ corresponding to $\psi(V(T'))$. 

Formally, given a finite sequence of integers $(x_n)_{n=0}^N$ with exactly $m$ distinct elements, there exists a unique function $f : \{x_n\} \to \{0,\dotsc, m-1\}$ so that $x_i < x_j$ if and only if $f(x_i) < f(x_j)$, $x_i = x_j$ if and only if $f(x_i) = f(x_j)$, and $x_i > x_j$ if and only if $f(x_i) > f(x_j)$. More precisely, $f$ is the function which groups $(x_n)$ by value, sorts them ascendingly, and maps each $x_n$ to its index in the resulting array. We define $(f(x_n))_{n=0}^N$ to be the \textit{shape} of the original sequence $(x_n)$. 

To see an example of this, consider the pair of trees at the top of Figure \ref{Fig:SubtreeExamples}. The populated values of the first line of the code for $T'$ are $(0,2,2,0)$, whereas the corresponding values in blue boxes in the code for $T$ are $(0,4,4,0)$. Both have shape $(0,1,1,0)$. Additionally, $(1,2,3,4,5,6)$ has the same shape as $(1,4,5,8,9,10)$, but does not have the same shape as $(1,1,2,3,4,5)$. 

The same shape in the first row of $P'$ and an appropriate subset of the first row of $P$ is not strong enough to determine if $T'$ is isomorphic to a subtree of $T$. Consider the example presented in Figure \ref{Fig:ThirdPropertyExample}. There is a slice of $P$ (given by the blue boxes around the appropriate columns of $P$) which has the same shape and same vertex colors as $P'$; however, the vertices giving rise to this slice of $P$ do not form a subtree of $T$ isomorphic to $T'$. So, it is necessary to place conditions on the columns of $P$ outside the columns corresponding to the subtree potentially isomorphic to $T'$. We leave the exact property statement to Theorem \ref{Thm:SubtreeIsomorphism}. The idea is that within a depth first search, suppose vertex $u$ is traversed before vertex $v$. With no additional information, it is impossible to know if $u$ is a vertex on the unique path between $v$ and the root. Consider the shape of the first row of $P'$ and a subset of the first row of $P$ aids in making this determination, but still does not provide enough information (see the example in Figure \ref{Fig:ThirdPropertyExample}). Instead we must consider the entire first row of $P$. 

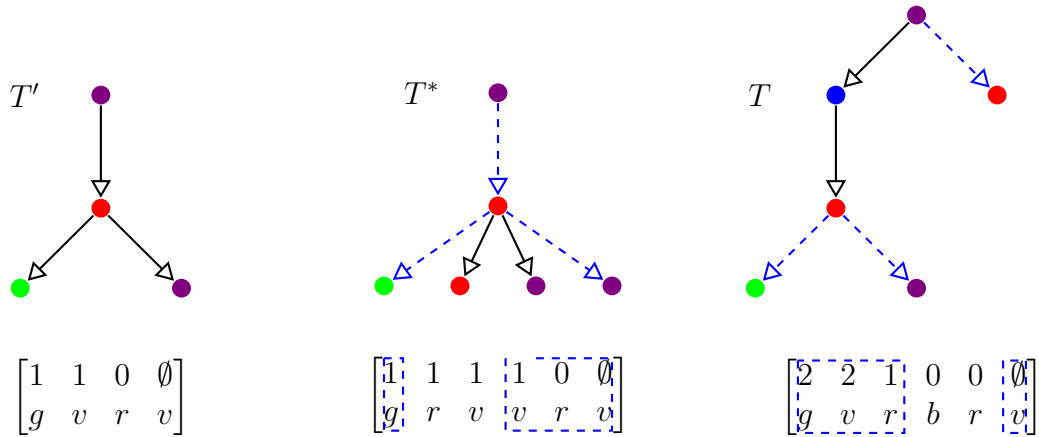
\begin{figure}[H]
\centering
\begin{minipage}{0.3\textwidth}
\centering
\begin{tikzpicture}[node distance = 15mm, thick, -open triangle 60, main/.style = {fill, circle, inner sep = 2.5pt}]
\node[main,violet] (1) {};
\node[main,red] (2) [below of=1] {};
\node[main,green] (3) [below left of=2] {};
\node[main,violet] (4) [below right of=2] {};
\node[main,white] (5) [above right of=1] {};

\draw (1) -- (2);
\draw (2) -- (3);
\draw (2) -- (4);

\node at (0,-4) {$\begin{bmatrix} 1&1&0&\emptyset \\ g&v&r&v \end{bmatrix}$};
\node at (-1,0) {$T'$};
\end{tikzpicture}
\end{minipage}
\begin{minipage}{0.3\textwidth}
\centering
\begin{tikzpicture}[node distance = 15mm, thick, main/.style = {fill, circle, inner sep = 2.5pt}]
\node[main,violet] (1) {};
\node[main,red] (2) [below of=1] {};
\node[main,green] (3) at (-1.5,-2.56066) {};
\node[main,red] (4) at (-0.5,-2.56066) {};
\node[main,violet] (5) at (0.5,-2.56066) {};
\node[main,violet] (6) at (1.5,-2.56066) {};

\node[main,white] (7) at (0,1) {};

\draw[-open triangle 60, blue,dashed] (1) -- (2);
\draw[-open triangle 60,blue,dashed] (2) -- (3);
\draw[-open triangle 60] (2) -- (4);
\draw[-open triangle 60] (2) -- (5);
\draw[-open triangle 60,blue,dashed] (2) -- (6);

\node at (0,-4) {$\begin{bmatrix} 1&1&1&1&0&\emptyset \\ g&r&v&v&r&v \end{bmatrix}$};
\node at (-1,0) {$T^*$};
\draw[thick,blue,dashed] (-1.5,-3.52) -- (-1.5,-4.48) -- (-1.25,-4.48) -- (-1.25,-3.52) -- (-1.5,-3.52);
\draw[thick,blue,dashed] (0.1,-3.52) -- (0.1,-4.48) -- (1.5,-4.48) -- (1.5,-3.52) -- (0.5,-3.52);
\end{tikzpicture}
\end{minipage}
\begin{minipage}{0.3\textwidth}
\centering
\begin{tikzpicture}[node distance = 15mm, thick, main/.style = {fill, circle, inner sep = 2.5pt}]
\node[main,blue] (1) {};
\node[main,red] (2) [below of=1] {};
\node[main,green] (3) [below left of=2] {};
\node[main,violet] (4) [below right of=2] {};
\node[main,violet] (5) [above right of=1] {};
\node[main,red] (6) [below right of=5] {};

\draw[-open triangle 60] (1) -- (2);
\draw[-open triangle 60,blue,dashed] (2) -- (3);
\draw[-open triangle 60,blue,dashed] (2) -- (4);
\draw[-open triangle 60] (5) -- (1);
\draw[-open triangle 60, blue,dashed] (5) -- (6);

\node at (1,-4) {$\begin{bmatrix} 2&2&1&0&0&\emptyset \\ g&v&r&b&r&v \end{bmatrix}$};
\node at (-1,0) {$T$};
\draw[thick,blue,dashed] (-0.5,-3.52) -- (-0.5,-4.48) -- (0.9,-4.48) -- (0.9,-3.52) -- (-0.5,-3.52);
\draw[thick,blue,dashed] (2.2,-3.52) -- (2.2,-4.48) -- (2.5,-4.48) -- (2.5,-3.52) -- (2.2,-3.52);
\end{tikzpicture}
\end{minipage}
\caption{Motivating example for the \textit{\textbf{incident edge}} property required for subtree isomorphism. Notice how the dashed blue edges of $T$ do not form a graph isomorphic to $T'$, even though the boxed portions of the VCPC for $T$ has the same shape as that of $T'$. }\label{Fig:ThirdPropertyExample}
\end{figure}

\subsection{Application}\label{Sec:Application}

We make use of the above isomorphism test when studying graphs built from United States freight data. The Department of Transportation's Bureau of Transportation Statistics \cite{FreightData} provides information detailing the weight and value of freight shipped within the US. These values are broken out by year from 2018 through 2023, and statistics for 42 different commodities are given. The country is partitioned geographically into 132 different regions, first by state (and the District of Columbia), then major cities are separated from the rest of the state. For example, Pennsylvania is divided into three regions, namely Pittsburgh, Philadelphia, and the rest of the state. Major cities that reach across multiple states (e.g., St. Louis, Kansas City, etc.) have separate entries for each state they influence. We aggregate these cities, reducing the count to 106 regions. 

For each combination of year and commodity, we build a graph whose vertices are the 106 regions, and edges connect two regions if a positive amount of the given commodity was shipped between the two regions in the given year. The edge weight is given by the weight (in tons) of the cargo moved. Vertices are partitioned into a few ``color'' classes. Regions which are entire states (or the portions of a state omitted by large cities) form a color class. Another class is a collection of 13 ``seaports'', large American ports that appear in the data. The remaining cities are then partitioned into four classes based on the presence of rail hubs. Union Pacific is the largest rail company in the United States. Their networks span the central and western portions of the continental US, while Norfolk Southern is the largest rail company that operates in the Eastern part of the country. The four classes for the remaining cities are given by the existence of rail hubs for each subset of these two companies. Overall, 25 non-seaport cities have a Norfolk Southern rail hub, 19 host a Union Pacific rail hub, and 6 are home to one of each. Notice that the graphs constructed frequently have monochromatic edges. 

In this way, $6\cdot 42 = 252$ graphs are constructed, one for each year-commodity pair, each on the 106 vertices defined above. Because of the construction, the graphs are quite dense, making them challenging to study. However, matroid partitioning (introduced in \cite{DarHarPhuPro2018}) identifies a spanning forest in each graph. By designating the seaports as rulers, the resulting matroid partition assigns a distinct component to each seaport. When applied to all 252 graphs, matroid partitioning forms 3242 components that contain a seaport (components not containing a seaport are ignored), the average containing 9 vertices. However, the distribution of tree orders is heavy-tailed. Many trees have few (less than five) vertices, so there are some with 75 or more vertices. For the computational convenience of this experiment, we restrict the collection of all components so that trees are only kept if they have at most 20 vertices. Thus, the collection of 3242 trees is reduced to a collection of 2843.

Common structure within these components indicates the typical influence a seaport city has on the surrounding local economy. To identify such structure, we partition the components into vertex-colored isomorphism classes, then for each pair of classes, $\mathcal{T}_1$ and $\mathcal{T}_2$, we determine if a representative of one class is isomorphic (as vertex-colored trees) to a subtree of a representative of the other. We perform this computation in two ways and compare the execution time. One implementation makes use of Python's networkx library, and the other is a Python implementation of the VCPC algorithms presented presently. Table \ref{Tab:FreightIso} summarizes the execution times (in seconds) of the algorithms. Note that except for the choice of isomorphism/subisomorphism algorithm, the same code is executed in both trials. 

\begin{table}[H]
\centering
\begin{tabular}{c | c c}
& networkx & Pr\"{u}fer coding  \\
\hline
isomorphism classes & 0.855 & 0.397 \\
subisomorphism & 216.476 & 117.075 \\
\end{tabular}
\caption{Run times (in seconds) to (1) compute the partition into vertex-colored isomorphism classes and (2) compute subtree isomorphisms. Note that the Pr\"{u}fer code isomorphism implementation time also includes the one-time work of computing the canonical ordering and VCPC for each tree.}\label{Tab:FreightIso}
\end{table}

While we are ultimately unable to provide complexity bounds on the Pr\"{u}fer code algorithms we introduce, the data of Table \ref{Tab:FreightIso} suggests the VCPC algorithms are faster than the networkx implementation, at least in this case. More experimentation is detailed in Section \ref{Sec:PythonExp}. 

We conclude the section by presenting a representative of the isomorphism class containing the most common structure, namely, the tree on at most 20 vertices containing the most components as a (not necessarily proper) subtree. The structure shown in Figure \ref{Fig:FreightExample} has 13 vertices and contains 2029 of the 2843 components as a (potentially improper) subtree. Vertices are labeled by their color class, where ``S'' denotes seaport (the ruling vertex), ``R'' denotes the ``rest of state'' or statewide regions, ``U'' and ``N'' denote the cities with only a Union Pacific or Norfolk Southern rail hub respectively, ``B'' denotes cities with hubs for both companies, and ``H'' denotes the cities without a hub, i.e., cities with a hub for neither of the companies considered. We take the ordering on color classes to be S $<$ B $<$ U $<$ N $<$ R $<$ H. 

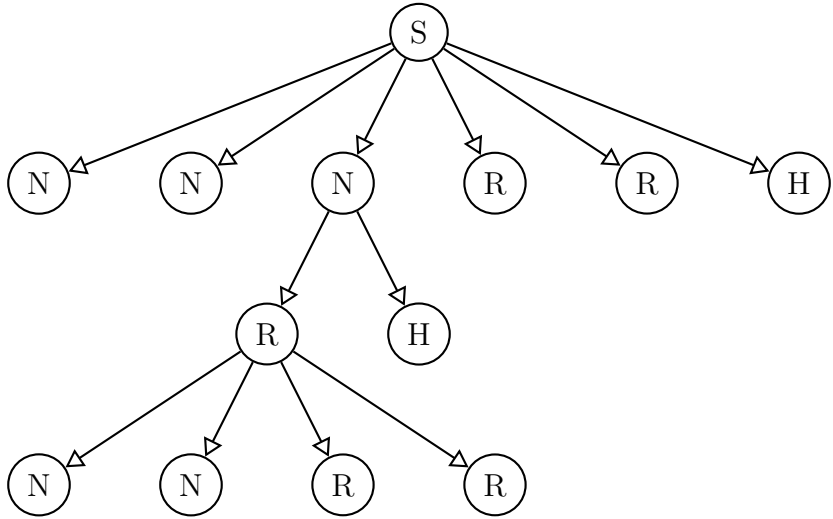
\begin{figure}[H]
\centering
\begin{tikzpicture}[thick, -open triangle 60, main/.style = {draw, circle}]
\node[main] (1) {S};
\node[main] (2) at (-5,-2) {N};
\node[main] (3) at (-3,-2) {N};
\node[main] (4) at (-1,-2) {N};
\node[main] (5) at (1,-2) {R};
\node[main] (6) at (3,-2) {R};
\node[main] (7) at (5,-2) {H};

\node[main] (8) at (-2,-4) {R};
\node[main] (9) at (0,-4) {H};

\node[main] (10) at (-5,-6) {N};
\node[main] (11) at (-3,-6) {N};
\node[main] (12) at (-1,-6) {R};
\node[main] (13) at (1,-6) {R};

\draw (1) -- (2);
\draw (1) -- (3);
\draw (1) -- (4);
\draw (1) -- (5);
\draw (1) -- (6);
\draw (1) -- (7);
\draw (4) -- (8);
\draw (4) -- (9);
\draw (8) -- (10);
\draw (8) -- (11);
\draw (8) -- (12);
\draw (8) -- (13);
\end{tikzpicture}
\caption{The most representative component on at most 20 vertices.}\label{Fig:FreightExample}
\end{figure}

\section{Formal Algorithms}\label{Sec:FormalAlgs}

Let $T$ be an arborescence rooted at vertex $r \in V(T)$, i.e., a directed tree where all edges are oriented away from the root. Furthermore, let the vertices of $T$ be colored by the function $c : V(T) \to \mathbb{N}$, where we do not require $c$ be a proper coloring of $T$, i.e., we permit monochromatic edges. For any $v \in V(T)$, define $T[v]$ to be the subtree of $T$ rooted at $v$ containing ($v$ and) all descendants of $v$. We aim to construct an ordering of $V(T)$ which acts as a ``standard'' depth first search of $T$. We accomplish this by sorting the children of $v$ for any $v\in V(T)$ based on $c$, and we call this search the \textit{Lexicographic Depth-First Search} of $T$, denoted $\mbox{L}\mathcal{D}(T)$. To this end, we simultaneously encode the graph structure and vertex colors of the descendants of $v$ for each $v \in V(T)$ in an array, called the \textit{Lexicographic Depth-First Search Array}, denoted $\mbox{L}\mathcal{D}_A(T[v])$. We present two algorithms to compute $\mbox{L}\mathcal{D}_A$, one a direct computation and the other recursive. 

We present our pseudocode making use of array and dictionary data structures. We start by storing all adjacency information for $T$ in a dictionary called $children$. For each $v \in V(T)$, if $v$ has at least one out-neighbor (meaning $v$ is the source for at least one directed edge), then $children$ contains $v$ as a key. The corresponding value is an array containing all out-neighbors of $v$ (in some arbitrary order).

Now we present a direct implementation which constructs $\mbox{L}\mathcal{D}_A(T)$. Within this function, $\mbox{L}\mathcal{D}_A(T[u])$ is computed for all vertices $u \in V(T)$ in the reverse order of some breadth first search. When $u \in V(T)$ is a leaf, define $\mbox{L}\mathcal{D}_A(T[u]) = [[]]$, i.e., an array containing an empty array as its only entry. When $u \in V(T)$ is not a leaf, $\mbox{L}\mathcal{D}_A(T[u])$ is constructed by calling Algorithm \ref{Alg:Sort} which sorts the children of $u$ by iteratively calling Algorithm \ref{Alg:GetMinVert}, which finds the smallest provided vertex based on their lexicographic depth first search array (these have already been computed because the vertices are in the reverse order of a breadth first search). 

Given vertices $u$ and $v$, what do we mean by $\mbox{L}\mathcal{D}_A(T[v]) < \mbox{L}\mathcal{D}_A(T[u])$? Ultimately, $\mbox{L}\mathcal{D}_A$ is an array of arrays where each entry is some color (where colors come from $\mathbb{N}$). To compare arrays $x$ and $y$, we define $x < y$ if and only if there exists index $i$ so that $x[j] = y[j]$ for all $j < i$ and either $x[i] < y[i]$ or $x[i]$ does not exist because $x$ has length $i$, but $y[i]$ does exist.

\begin{algorithm}[H]
\caption{$get\_min\_vertex(verts, c, L\mathcal{D}_A)$ - return the minimum vertex in $verts$}\label{Alg:GetMinVert}
\begin{algorithmic}
\STATE $min\_vert \gets$ the zeroth vertex in $verts$
\STATE remove the zeroth entry of $verts$
\FOR{$v$ in $verts$}
\IF{$c(v) < c(min\_vert)$}
\STATE $min\_vert \gets v$
\ELSIF{if $c(v) = c(min\_vert)$ and $L\mathcal{D}_A[T[v]] < L\mathcal{D}_A[T[min\_vert]]$}
\STATE $min\_vert \gets v$
\ENDIF
\ENDFOR
\RETURN $min\_vert$
\end{algorithmic}
\end{algorithm}

\begin{algorithm}[H]
\caption{$sort\_siblings(verts, c, L\mathcal{D}_A)$ - sorting siblings by color and $L\mathcal{D}_A$}\label{Alg:Sort}
\begin{algorithmic}
\STATE $sorted\_verts \gets []$ 
\WHILE{$verts$ is not empty}
\STATE $min\_vert \gets get\_min\_vertex(verts, c, L\mathcal{D}_A)$
\STATE append $min\_vert$ to $sorted\_verts$
\STATE remove $min\_vert$ from $verts$ 
\ENDWHILE
\RETURN $sorted\_verts$
\end{algorithmic}
\end{algorithm}

\begin{algorithm}[H]
\caption{get\_$\mbox{L}\mathcal{D}_A$($children$, $r$, $c$) - computing $\mbox{L}\mathcal{D}_A(T)$}\label{Alg:DirectDef}
\begin{algorithmic}
\STATE $B \gets$ an array containing a breadth-first search ordering of $V(T)$. 
\STATE $L\mathcal{D}_A \gets \{\}$ 

\FOR{$v$ in the reverse ordering of $B$}
\STATE $L\mathcal{D}_A[v] \gets []$
\IF{$v$ is a key in $children$}
\STATE $sn \gets$ sort the out-neighbors of $v$ with $sort\_siblings(children[v], c, L\mathcal{D}_A)$ 
\FOR{$n$ in $sn$}
\STATE append the color of $n$ to the zeroth array in $L\mathcal{D}_A[v]$
\STATE append all elements of $L\mathcal{D}_A[n]$ to $L\mathcal{D}_A[v]$
\ENDFOR
\ENDIF
\ENDFOR

\RETURN $L\mathcal{D}_A[r]$
\end{algorithmic}
\end{algorithm}


The following figure provides an example of a tree $T$, where each vertex $v$ is labeled with $\mbox{L}\mathcal{D}_A(T[v])$. Suppose the colors $\{0,1,2\}$ are represented as blue, green, and yellow respectively. 

\begin{figure}[H]
\centering
\begin{tikzpicture}[thick, -open triangle 60, main/.style = {node distance = 15mm, fill, circle, inner sep = 2.5pt}, label/.style = {node distance = 4mm}]
\node[main,blue] (1) {};
\node[main,yellow] (2) at (-2,-1) {};
\node[main,green] (3) at (2,-1) {};
\node[main,green] (4) [below left of=2] {};
\node[main,blue] (5) [below right of=2] {};
\node[main,green] (6) [below left of=3] {};
\node[main,green] (7) [below right of=3] {};
\node[main,blue] (8) [below left of=7] {};
\node[main,yellow] (9) at (3.06066,-3.12132) {};
\node[main,blue] (10) [below right of=7] {};

\node[label] [above of=1] {$[[1,2], [1,1], [], [0,0,2], [], [], [], [0,1], [], []]$};
\node[label] at (-3.25,-1) {$[[0,1], [], []]$};
\node[label] at (4.35,-1) {$[[1,1], [], [0,0,2], [], [], []]$};
\node[label] [below of=4] {$[[]]$};
\node[label] [below of=5] {$[[]]$};
\node[label] [below of=6] {$[[]]$};
\node[label] at (4.65,-2.06066) {$[[0,0,2], [], [], []]$};
\node[label] [below of=8] {$[[]]$};
\node[label] [below of=9] {$[[]]$};
\node[label] [below of=10] {$[[]]$};

\draw (1) -- (2);
\draw (1) -- (3);
\draw (2) -- (4);
\draw (2) -- (5);
\draw (3) -- (6);
\draw (3) -- (7);
\draw (7) -- (8);
\draw (7) -- (9);
\draw (7) -- (10);
\end{tikzpicture}
\caption{$\mbox{L}\mathcal{D}_A$ Example on tree $T$}\label{Fig:LDLExample}
\end{figure}

This example illustrates one limitation of the $\mbox{L}\mathcal{D}_A$ construction, namely that $\mbox{L}\mathcal{D}_A(T[v])$ encodes information about the colors and substructure of $T[v]$ but not the color of $v$ itself. So, it is impossible to reconstruct $c : V(T) \to \mathbb{N}$ from $\mbox{L}\mathcal{D}_A(T)$ since information about $c(r)$ is omitted. One way around this is to add a new artificial root node $r'$, with an edge from $r$ to $r'$, and $c(r')$ set arbitrarily. Then compute $\mbox{L}\mathcal{D}_A(T')$, which will be $[[c(r)]] + \mbox{L}\mathcal{D}_A(T)$, where ``$+$'' is concatenation. An example is illustrated in Figure \ref{Fig:LDLPseudoRootExample}. We arbitrarily chose $c_{T'}(r') = y$. 

\begin{figure}[H]
\centering
\begin{tikzpicture}[thick, -open triangle 60, main/.style = {node distance = 15mm, fill, circle, inner sep = 2.5pt}, label/.style = {node distance = 4mm}]
\node[main,blue] (1) {};
\node[main,yellow] (2) at (-2,-1) {};
\node[main,green] (3) at (2,-1) {};
\node[main,green] (4) [below left of=2] {};
\node[main,blue] (5) [below right of=2] {};
\node[main,green] (6) [below left of=3] {};
\node[main,green] (7) [below right of=3] {};
\node[main,blue] (8) [below left of=7] {};
\node[main,yellow] (9) at (3.06066,-3.12132) {};
\node[main,blue] (10) [below right of=7] {};
\node[main,yellow] (11) [above of=1] {};

\node[label] at (4,0) {$[[1,2], [1,1], [], [0,0,2], [], [], [], [0,1], [], []]$};
\node[label] at (-3.25,-1) {$[[0,1], [], []]$};
\node[label] at (4.35,-1) {$[[1,1], [], [0,0,2], [], [], []]$};
\node[label] [below of=4] {$[[]]$};
\node[label] [below of=5] {$[[]]$};
\node[label] [below of=6] {$[[]]$};
\node[label] at (4.3,-2.06066) {$[[0,1], [], []]$};
\node[label] [below of=8] {$[[]]$};
\node[label] [below of=9] {$[[]]$};
\node[label] [below of=10] {$[[]]$};
\node[label] [above of=11] {$[[0], [1,2], [1,1], [], [0,0,2], [], [], [], [0,1], [], []]$};

\draw (11) -- (1);
\draw (1) -- (2);
\draw (1) -- (3);
\draw (2) -- (4);
\draw (2) -- (5);
\draw (3) -- (6);
\draw (3) -- (7);
\draw (7) -- (8);
\draw (7) -- (9);
\draw (7) -- (10);
\end{tikzpicture}
\caption{$\mbox{L}\mathcal{D}_A$ Example on tree $T'$ to obtain $\overline{\mbox{L}\mathcal{D}_A}(T)$}\label{Fig:LDLPseudoRootExample}
\end{figure}
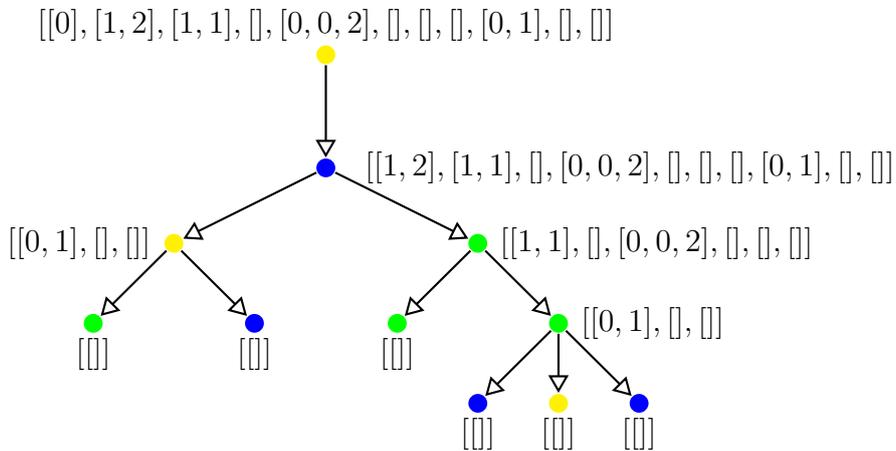

In this way, we call $\mbox{L}\mathcal{D}_A(T')$ the \textit{Full Lexicographic Depth-First Search Array} of $T$, denoted $\overline{\mbox{L}\mathcal{D}_A}(T)$, and from $\overline{\mbox{L}\mathcal{D}_A}(T)$ it is possible to reconstruct $T$ together with the coloring function $c : V(T) \to\mathbb{N}$. The following algorithm provides the reconstruction, relying on none of the algorithms defined above. It returns both a dictionary of neighbors (where we label vertices $0$ through $n-1$) and the coloring function, which is enough information to uniquely identify $T$. 

\begin{algorithm}[H]
\caption{$reconstruct\_tree(\overline{L\mathcal{D}_A}, co)$ - reconstruct $T$ from $\overline{\mbox{L}\mathcal{D}_A}(T)$}\label{Alg:RecomputeTree}
\begin{algorithmic}
\STATE $vertex\_count \gets 0$
\STATE $c \gets$ an empty dictionary
\FOR{each array $a$ in $\overline{L\mathcal{D}_A}$}
\FOR{each entry $e$ of $a$}
\STATE $c[vertex\_count] \gets e$, which is the color of some new vertex
\STATE replace $e$ with $vertex\_count$
\STATE increment $vertex\_count$
\ENDFOR
\ENDFOR
\STATE
\STATE $children \gets $ an empty dictionary 
\STATE $A \gets $ an array containing the zeroth entry of $\overline{L\mathcal{D}_A}$
\STATE remove the zeroth element from $\overline{L\mathcal{D}_A}$
\FOR{each array $a$ in $\overline{L\mathcal{D}_A}$}
\STATE $cur\_vert \gets A[0][0]$
\STATE remove $cur\_vert$ from $A[0]$
\IF{$A[0] = []$}
\STATE remove the zeroth element of $A$
\ENDIF
\IF{$a \neq []$}
\STATE add the (key, value) pair $(cur\_vert, l)$ to $children$
\STATE prepend the list containing the single element $a$ to $A$
\ENDIF
\ENDFOR 
\STATE
\RETURN $children$, $color$
\end{algorithmic}
\end{algorithm}

From the algorithm for $\mbox{L}\mathcal{D}_A$ it is straightforward to obtain the \textit{Lexicographic Depth First Search} of $T$, denoted $\mbox{L}\mathcal{D}(T)$. Simply begin a depth first search at the root, and traverse children in the order that preserves $\mbox{L}\mathcal{D}_A$. Frequently we use $\phi : V(T) \to \{0,\dotsc,n-1\}$ to denote $\mbox{L}\mathcal{D}(T)$, where if $u, v \in V(T)$ then $u$ is traversed before $v$ in $\mbox{L}\mathcal{D}(T)$ if and only if $\phi(u) < \phi(v)$. Here are a few examples of $\mbox{L}\mathcal{D}(T)$, where a drawing of $T$ is provided on the left and each vertex on the right is labeled by $\phi$ in the redrawing of $T$ that realizes the $\mbox{L}\mathcal{D}$. The colors utilized (black, green, red, and yellow) are ordered lexicographically. 

\begin{figure}[H]
\centering
\begin{minipage}{0.45\textwidth}
\centering
\begin{tikzpicture}[node distance = 15mm, thick, -open triangle 60, main/.style = {fill, circle, inner sep = 2.5pt}]
\node[main,black] (1) {};
\node[main,yellow] (2) at (-2,-1) {};
\node[main,green] (3) at (2,-1) {};
\node[main,red] (4) [below left of=2] {};
\node[main,green] (5) [below right of=2] {};
\node[main,yellow] (6) [below left of=3] {};
\node[main,black] (7) [below right of=3] {};

\draw (1) -- (2);
\draw (1) -- (3);
\draw (2) -- (4);
\draw (2) -- (5);
\draw (3) -- (6);
\draw (3) -- (7);
\end{tikzpicture}
\end{minipage}
\begin{minipage}{0.05\textwidth}
$$\rightarrow$$
\end{minipage}
\begin{minipage}{0.45\textwidth}
\centering
\begin{tikzpicture}[node distance = 15mm, thick, -open triangle 60, main/.style = {fill, circle, inner sep = 2.5pt}]
\node[main,black] (1) {};
\node[main,green] (2) at (-2,-1) {};
\node[main,yellow] (3) at (2,-1) {};
\node[main,black] (4) [below left of=2] {};
\node[main,yellow] (5) [below right of=2] {};
\node[main,green] (6) [below left of=3] {};
\node[main,red] (7) [below right of=3] {};

\node[] [left of=1, node distance = 5mm] {0};
\node[] [left of=2, node distance = 4mm] {1};
\node[] [left of=3, node distance = 5mm] {4};
\node[] [left of=4, node distance = 4mm] {2};
\node[] [left of=5, node distance = 4mm] {3};
\node[] [left of=6, node distance = 4mm] {5};
\node[] [left of=7, node distance = 4mm] {6};

\draw (1) -- (2);
\draw (1) -- (3);
\draw (2) -- (4);
\draw (2) -- (5);
\draw (3) -- (6);
\draw (3) -- (7);
\end{tikzpicture}
\end{minipage}
\begin{minipage}{0.45\textwidth}
\centering
\begin{tikzpicture}[node distance = 15mm, thick, -open triangle 60, main/.style = {fill, circle, inner sep = 2.5pt}]
\node[main,black] (1) {};
\node[main,red] (2) at (-2,-1) {};
\node[main,red] (3) at (2,-1) {};
\node[main,yellow] (4) [below left of=2] {};
\node[main,green] (5) [below right of=2] {};
\node[main,yellow] (6) [below left of=3] {};
\node[main,black] (7) [below right of=3] {};

\draw (1) -- (2);
\draw (1) -- (3);
\draw (2) -- (4);
\draw (2) -- (5);
\draw (3) -- (6);
\draw (3) -- (7);
\end{tikzpicture}
\end{minipage}
\begin{minipage}{0.05\textwidth}
$$\rightarrow$$
\end{minipage}
\begin{minipage}{0.45\textwidth}
\centering
\begin{tikzpicture}[node distance = 15mm, thick, -open triangle 60, main/.style = {fill, circle, inner sep = 2.5pt}]
\node[main,black] (1) {};
\node[main,red] (2) at (-2,-1) {};
\node[main,red] (3) at (2,-1) {};
\node[main,black] (4) [below left of=2] {};
\node[main,yellow] (5) [below right of=2] {};
\node[main,green] (6) [below left of=3] {};
\node[main,yellow] (7) [below right of=3] {};

\node[] [left of=1, node distance = 5mm] {0};
\node[] [left of=2, node distance = 4mm] {1};
\node[] [left of=3, node distance = 5mm] {4};
\node[] [left of=4, node distance = 4mm] {2};
\node[] [left of=5, node distance = 4mm] {3};
\node[] [left of=6, node distance = 4mm] {5};
\node[] [left of=7, node distance = 4mm] {6};

\draw (1) -- (2);
\draw (1) -- (3);
\draw (2) -- (4);
\draw (2) -- (5);
\draw (3) -- (6);
\draw (3) -- (7);
\end{tikzpicture}
\end{minipage}
\begin{minipage}{0.45\textwidth}
\centering
\begin{tikzpicture}[node distance = 15mm, thick, -open triangle 60, main/.style = {fill, circle, inner sep = 2.5pt}]
\node[main,black] (1) {};
\node[main,red] (2) at (-2,-1) {};
\node[main,red] (3) at (2,-1) {};
\node[main,green] (4) [below left of=2] {};
\node[main,green] (5) [below right of=2] {};
\node[main,green] (6) [below left of=3] {};
\node[main,green] (7) [below right of=3] {};

\draw (1) -- (2);
\draw (1) -- (3);
\draw (2) -- (4);
\draw (2) -- (5);
\draw (3) -- (6);
\draw (3) -- (7);
\end{tikzpicture}
\end{minipage}
\begin{minipage}{0.05\textwidth}
$$\rightarrow$$
\end{minipage}
\begin{minipage}{0.45\textwidth}
\centering
\begin{tikzpicture}[node distance = 15mm, thick, -open triangle 60, main/.style = {fill, circle, inner sep = 2.5pt}]
\node[main,black] (1) {};
\node[main,red] (2) at (-2,-1) {};
\node[main,red] (3) at (2,-1) {};
\node[main,green] (4) [below left of=2] {};
\node[main,green] (5) [below right of=2] {};
\node[main,green] (6) [below left of=3] {};
\node[main,green] (7) [below right of=3] {};

\node[] [left of=1, node distance = 5mm] {0};
\node[] [left of=2, node distance = 4mm] {1};
\node[] [left of=3, node distance = 5mm] {4};
\node[] [left of=4, node distance = 4mm] {2};
\node[] [left of=5, node distance = 4mm] {3};
\node[] [left of=6, node distance = 4mm] {5};
\node[] [left of=7, node distance = 4mm] {6};

\draw (1) -- (2);
\draw (1) -- (3);
\draw (2) -- (4);
\draw (2) -- (5);
\draw (3) -- (6);
\draw (3) -- (7);
\end{tikzpicture}
\end{minipage}
\caption{Examples of $\mbox{L}\mathcal{D}(T)$}\label{Fig:ColorInducedLabelExample}
\end{figure}
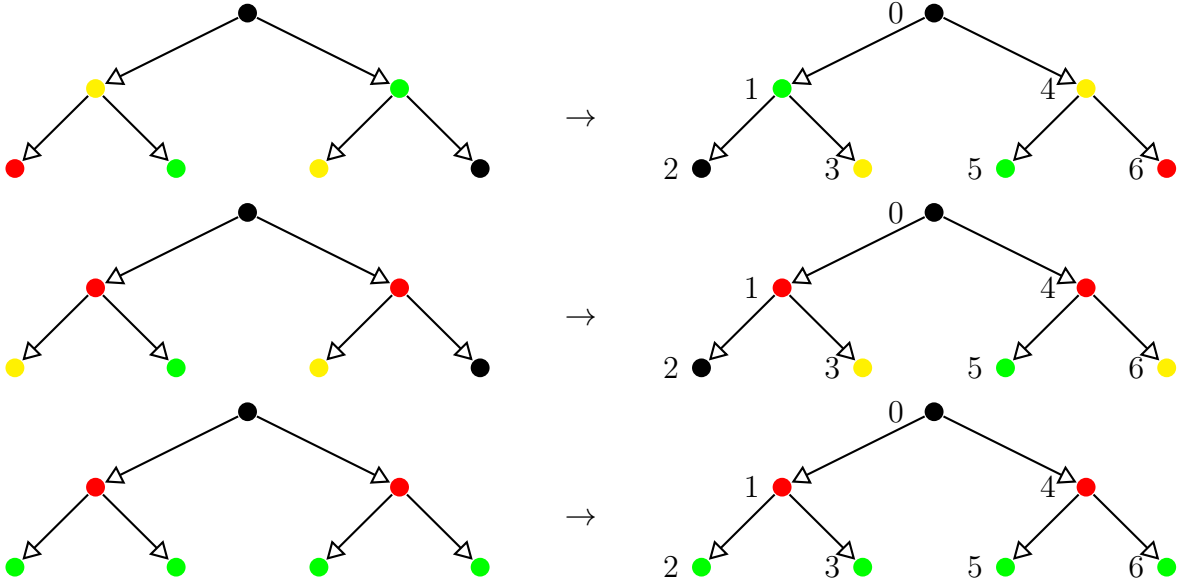

Now we progress to the underlying theory. Throughout the rest of the section, assume vertex colors are given by $\mathbb{N}$. Furthermore, let $\mathcal{T}_{\mathbb{N}}$ be the collection of all finite arborescences with vertices colored by $\mathbb{N}$ (where we allow monochromatic edges). For $T \in \mathcal{T}_\mathbb{N}$, let $c_T : V(T) \to \mathbb{N}$ be the coloring of vertices in $T$. 

\begin{prop}
$\overline{\mbox{L}\mathcal{D}}_L$ defines an equivalence relation on $\mathcal{T}_{\mathbb{N}}$ where for pair $S, T \in \mathcal{T}_\mathbb{N}$, $S$ and $T$ are equivalent if and only if 
$\overline{\mbox{L}\mathcal{D}}(S)$ and  $\overline{\mbox{L}\mathcal{D}}(T)$ are identical. 
\end{prop}
\begin{proof}
Clearly $\overline{\mbox{L}\mathcal{D}}_L$ defines a reflexive and symmetric relation. Transitivity is given by the transitivity of equality for arrays. 
\end{proof}

\begin{prop}
The $\overline{\mbox{L}\mathcal{D}}_L$ equivalence relation defined on $\mathcal{T}_{\mathbb{N}}$ is the same as the equivalence relation defined by isomorphism as vertex-colored arborescence. In other words, if $T, T' \in \mathcal{T}_{\mathbb{N}}$, then $T$ and $T'$ are isomorphic as vertex-colored arborescences if and only if the arrays $\overline{\mbox{L}\mathcal{D}_A}(T)$ and $\overline{\mbox{L}\mathcal{D}_A}(T')$ are identical.
\end{prop}
\begin{proof}
$(\Rightarrow):$ Suppose $T$ and $T'$ are isomorphic as vertex-colored arborescences, and let $\psi : V(T) \to V(T')$ define the isomorphism. Let $v \in V(T)$ and $v' \in V(T')$ so that $v' = \psi(v)$. Since $\psi$ is a vertex-colored isomorphism, $\overline{\mbox{L}\mathcal{D}_A}(T[v]) = \overline{\mbox{L}\mathcal{D}_A}(T'[v'])$, and the desired result follows. 

$(\Leftarrow):$ Suppose instead that $\overline{\mbox{L}\mathcal{D}_A}(T) = \overline{\mbox{L}\mathcal{D}_A}(T')$. Then the output of Algorithm \ref{Alg:RecomputeTree} is the same for both $T$ and $T'$.
\end{proof}

\begin{prop}\label{Thm:TotalOrderEquivRelation}
The ordering of $\{\overline{\mbox{L}\mathcal{D}_A}(T)\}_{T \in \mathcal{T}_{\mathbb{N}}}$ defined within Algorithm \ref{Alg:GetMinVert} induces a total ordering on the equivalence classes of $\mathcal{T}_{\mathbb{N}}$ under the $\overline{\mbox{L}\mathcal{D}}_L$ relation.
\end{prop}
\begin{proof}
The total ordering of isomorphism classes under $\overline{\mbox{L}\mathcal{D}_A}$ defines $T_1 < T_2$ if and only if $\overline{\mbox{L}\mathcal{D}_A}(T_1) < \overline{\mbox{L}\mathcal{D}_A}(T_2)$ for vertex colored arborescences $T_1, T_2 \in \mathcal{T}_{\mathbb{N}}$, where ``$<$'' for $\overline{\mbox{L}\mathcal{D}_A}$ is defined directly before Algorithm \ref{Alg:GetMinVert}. 
\end{proof}

\begin{cor}
The ordering of $\{\overline{\mbox{L}\mathcal{D}_A}(T)\}_{T \in \mathcal{T}_{\mathbb{N}}}$ defined within Algorithm \ref{Alg:GetMinVert} induces a pre-ordering $<_{\mathcal{T}_{\mathbb{N}}}$ on $\mathcal{T}_{\mathbb{N}}$.
\end{cor}
\begin{proof}
Extend the relation given in the proof of Proposition \ref{Thm:TotalOrderEquivRelation} so that $T_1 \leq T_2$ if and only if $\overline{\mbox{L}\mathcal{D}_A}(T_1) \leq \overline{\mbox{L}\mathcal{D}_A}(T_2)$ for vertex colored arborescences $T_1$ and $T_2$, which is clearly a reflexive and transitive relation on $\mathcal{T}_\mathbb{N}$. 
\end{proof}

\begin{prop}
Let $T \in \mathcal{T}_\mathbb{N}$. Then there exists a total ordering $<_T$ of $V(T)$ so that $<_{T}$ is a depth-first search of $V(T)$ and for any two siblings $u$ and $v$, if $u <_T v$, then $T[u] \leq_{\mathcal{T}_{\mathbb{N}}} T[v]$.
\end{prop}
\begin{proof}
One such well ordering is given by $\mbox{L}\mathcal{D}$.
\end{proof}

\begin{prop}\label{Thm:UniqueLabel}
Let $T \in \mathcal{T}_\mathbb{N}$ have totally ordered vertex set $V(T)$ given by $<_V$. Then there exists a \underline{unique} total ordering $<_T$ of $V(T)$ so that $<_{T}$ is a depth-first search of $T$ and for any two siblings $u$ and $v$, $u <_T v$ if and only if $T[u] <_{\mathcal{T}_{\mathbb{N}}} T[v]$ or $T[u] =_{\mathcal{T}_{\mathbb{N}}} T[v]$ and $u <_V v$.
\end{prop}
\begin{proof}
The ordering is defined by $\mbox{L}\mathcal{D}$, provided the sorting algorithm in Algorithm \ref{Alg:GetMinVert} is amended so that if $u,v \in T$ are siblings with $\overline{\mbox{L}\mathcal{D}_A}(T[u]) = \overline{\mbox{L}\mathcal{D}_A}(T[v])$, then $u <_T v$ if and only if $u <_V v$. 
\end{proof}

The previous definition aims to distinguish between two vertex-colored arborescences whenever they are not isomorphic. The following result justifies this property of the labeling, but first we formalize the definition of a vertex-colored isomorphism. 

\begin{definition}
Let $G$ and $G'$ be two graphs. Then $\psi : G \to G'$ is a vertex-colored isomorphism when $\psi$ is a graph isomorphism (preserves adjacency and non-adjacency) so that the color of $v \in G$ matches the color of $\psi(v)$ in $G'$. 
\end{definition}

\begin{cor}\label{Prop:CompositionOfPhisIsIsomorphism}
Let $T$ and $T'$ be two vertex-colored arborescences with $\phi : V(T) \to \{0,\dotsc, |V(T)| - 1\}$ and $\phi' : V(T') \to \{0,\dotsc, |V(T')| - 1\}$ be the vertex labelings which represent the $\mbox{L}\mathcal{D}$ ordering. Then $T$ and $T'$ are isomorphic as vertex colored arborescences if and only if $\phi^{-1} \circ\phi'$ and $\phi'^{-1}\circ\phi$ are color preserving isomorphisms. 
\end{cor}
\begin{proof}
This is a corollary to Proposition \ref{Thm:UniqueLabel}. The backwards implication is trivial. For the forward implication, $T$ and $T'$ isomorphic as vertex-colored arborescences implies $\overline{\mbox{L}\mathcal{D}_A}(T) = \overline{\mbox{L}\mathcal{D}_A}(T')$. Thus, the vertex-colored depth first searches of $T$ and $T'$ are identical, meaning $\phi^{-1} \circ \phi'$ preserves vertex colors. Furthermore, $\overline{\mbox{L}\mathcal{D}_A}(T) = \overline{\mbox{L}\mathcal{D}_A}(T')$ implies the sorted array of vertex colors of children of $\phi^{-1}(i)$ and $\phi'^{-1}(i)$ agree for any $0 \leq i \leq |V(T)| - 1$. Thus, $\phi^{-1} \circ\phi'$ and $\phi'^{-1}\circ\phi$ are color preserving isomorphisms.
\end{proof}

Now we formally define VCPCs. Based on the vertex-labeling defined by $\mbox{L}\mathcal{D}$ the following definition closely resembles the third tree encoding method presented by Neville (\cite{Nev1953}). 

\begin{definition}\label{Def:ColoredPruferCodes}
Let $T$ be an arborescence on $n$ vertices. Let $\phi: V(T) \to \{0,\dotsc, |V(T)| - 1\}$ be the vertex labeling which encodes $\mbox{L}\mathcal{D}(T)$. Then, the \textbf{VCPC} for $T$ is a $2 \times n$ array of entries so the first row contains the standard Pr\"{u}fer code on $n-1$ symbols (we prune a vertex even when only two remain), where leaves are vertices with out-degree zero. The second row contains the color of the vertex pruned at each step. In the last column, record $\emptyset$ in the first row and the color of the root of $T$ in the second row. 
\end{definition}

Let $P$ denote the VCPC for $T$. Then the standard Pr\"{u}fer code with $n-2$ symbols on that vertex labeling is the same as the first $n-2$ symbols in the first row of $P$. As a result, the isomorphism provided by Pr\"{u}fer codes on labeled trees applies to the augmented Pr\"{u}fer code we define in the first row of $P$. We use this fact throughout. 

The following figure illustrates an example of how to construct the VCPC, where an arbitrary drawing of $T$ is shown on the left, the drawing which realizes $\mbox{L}\mathcal{D}(T)$ with vertex labeling $\phi$ is shown on the right, and the constructed Pr\"{u}fer code is listed underneath. 

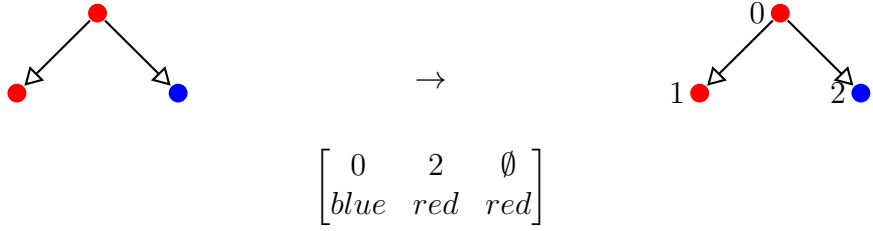
\begin{figure}[H]
\centering
\begin{minipage}{0.45\textwidth}
\centering
\begin{tikzpicture}[node distance = 15mm, thick, -open triangle 60, main/.style = {fill, circle, inner sep = 2.5pt}]
      \node[main,red] (1) {};
      \node[main,red] (2) [below left of=1] {};
      \node[main,blue] (3) [below right of=1] {};

      \draw (1) -- (2);
      \draw (1) -- (3);
\end{tikzpicture}
\end{minipage}
\begin{minipage}{0.05\textwidth}
$$
\rightarrow
$$
\end{minipage}
\begin{minipage}{0.45\textwidth}
\centering
\begin{tikzpicture}[node distance = 15mm, thick, -open triangle 60, main/.style = {fill, circle, inner sep = 2.5pt}]
\node[main,red] (1) {};
\node[main,red] (2) [below left of=1] {};
\node[main,blue] (3) [below right of=1] {};
\node[] [left of=1, node distance = 3mm] {0};
\node[] [left of=2, node distance = 3mm] {1};
\node[] [left of=3, node distance = 3mm] {2};

\draw (1) -- (2);
\draw (1) -- (3);
\end{tikzpicture}
\end{minipage}
$$
\begin{bmatrix}
0&2&\emptyset \\
blue&red&red
\end{bmatrix}
$$
\caption{\textit{VCPC example}}\label{Fig:SmallPruferExample}
\end{figure}


Throughout the rest of the manuscript, we adopt the following notational conventions. We let $T$ and $T'$ denote vertex-colored arborescences on $n$ and $n'$ vertices respectively where $n'\leq n$. Let $r \in V(T)$ and $r' \in V(T')$ be the roots. Use $c : V(T) \to \mathbb{N}$ and $c' : V(T') \to \mathbb{N}$ to denote the colors of $V(T)$ and $V(T')$ respectively. Furthermore, let $\phi$ and $\phi'$ be the vertex labelings which preserve $\mbox{L}\mathcal{D}(T)$ and $\mbox{L}\mathcal{D}(T')$ respectively. Let $P$ and $P'$ be the VCPC for $T$ and $T'$ respectively. Furthermore, we let $\{p_i\}_{i=0}^{n-1}$ and $\{p'_i\}_{i=0}^{n'-1}$ denote the first rows of $P$ and $P'$, while $\{c_i\}_{i=0}^{n-1}$ and $\{c_i'\}_{i=0}^{n'-1}$ denote the second rows of $P$ and $P'$. Lastly, let $v_i \in V(T)$ and $v_i' \in V(T')$ be the vertex pruned at the $i$th step of the construction of $P$ and $P'$ respectively, with $u_i$ the parent of $v_i$ for $i < n$ and $u_i'$ the parent of $v_i'$ for $i < n'$.

\section{Relationships Among Vertex-Colored Arborescences}\label{Sec:CompClassify}

Here we present the formal statements and proofs of isomorphism results in relation to VCPCs. We start with one useful definition. 

\begin{definition}\label{Def:BranchPartition}
Let $T$ be an arborescence. Let $u \in V(T)$ and $b \in \N$ so that $\{v_i\}_{i=0}^{b-1}$ are the children of $u$. Then the \textbf{branch partition} of $T[u]$, denoted $\mathcal{B}_u = \{B_i\}_{i=0}^{b-1}$ is the partition of $V(T[u]) \setminus \{u\}$ so that $B_i = T[v_i]$. 
\end{definition}

\begin{definition}\label{Def:Shape}
Let $X = (x_n)_{n=0}^N$ be a finite sequence of real numbers so that $|\{x_n\}_{n=0}^N| = m$. Then the \textbf{shape} of $X$, denoted $s(X)$, is $s(X) = \{f(x_n)\}_{n=0}^N$, where $f : \{x_n\}_{n=0}^N \to \{0,\dotsc, m-1\}$ is the unique order-preserving bijection, i.e., the unique bijection satisfying $x_i < x_j$ if and only if $f(x_i) < f(x_j)$. 
\end{definition}


\begin{theorem}\label{Prop:VertexColoredPruferIsomorphism}
There exists a vertex-colored isomorphism between arborescences $T$ and $T'$ if and only if $P = P'$. 
\end{theorem}
\begin{proof}
$(\Rightarrow):$ 
Suppose there exists a vertex-colored isomorphism between $T$ and $T'$. Then the depth-first searches of $T$ and $T'$ given by $\mbox{L}\mathcal{D}$ are identical, so $P$ and $P'$ are the same.

$(\Leftarrow):$ Suppose that $P$ and $P'$ are identical arrays of size $2 \times n$. 
Since the first rows of $P$ and $P'$ are the same, the depth first searches for $T$ and $T'$ given by $\mbox{L}\mathcal{D}$ are the same. Therefore, if $\phi : V(T) \to \{0,\dotsc, n-1\}$ and $\phi': V(T') \to \{0,\dotsc,n-1\}$ are the vertex labelings defined by $\mbox{L}\mathcal{D}$ and $c : V(T) \to \mathbb{N}$ and $c' : V(T') \to \mathbb{N}$ are the vertex colorings of $T$ and $T'$ respectively, then the equality of second rows of $P$ and $P'$ imply $c(\phi^{-1}(i)) = c'(\phi'^{-1}(i))$ for each $0 \leq i \leq n-1$. Thus, $\psi : V(T) \to V(T')$ defined by $\psi(v) = \phi'^{-1}(\phi(v))$ is a vertex-colored isomorphism between $T$ and $T'$. 
\end{proof}

The previous result totally characterizes when two vertex-colored arborescences are isomorphic. Next, we present a lemma describing how to identify adjacencies within the VCPC, and we ultimately use this in the subsequent theorem detailing conditions equivalent to subtree isomorphism. Recall that $p_i$ is used to denote the $i$th entry of the first row of the VCPC $P$.

\begin{lemma}\label{Lem:PruferAdjacency}
Let $i < j$. Then $(v_j, v_i) \in E(T)$ if and only if $p_j < p_i$ and for any $k \in\N$ satisfying $i < k < j$, $p_k \geq p_i$. 
\end{lemma}
\begin{proof}
Recall that $\phi : V(T) \to \{0,\dotsc, |V(T)| - 1\}$ is the vertex labeling which represents the $\mbox{L}\mathcal{D}$ ordering. 

$(\Rightarrow):$ Suppose that $(v_j, v_i) \in E(T)$. Since $v_j$ is the parent of $v_i$, $\phi(v_j) < \phi(v_i)$. Since $\phi(v_j) = p_i$ and $p_j < \phi(v_j)$, we have $p_j < p_i$. 
Suppose $i < k < j$ so that $v_k$ is pruned after $v_i$ and before $v_j$. By way of contradiction, further suppose that $k$ is the largest index between $i$ and $j$ so that $v_k$ is not a descendant of $v_j$. Since $k > i$, $\phi(v_i) < \phi(v_k)$, so $v_k$ not a descendant of $v_j$ implies 
$\phi(v_k) > \phi(u)$ for any $u \in V(T[v_j])$. Since $k < j$ is the largest index less than $j$ so that $v_k$ is not a descendant of $v_j$, $v_{k+1}$ is either $v_j$ or a descendant of $v_j$. Thus, the pruning of $v_k$ does not change the out-degree of $v_{k+1}$, so $v_{k+1}$ has out-degree zero when $v_k$ is pruned. However, $\phi(v_{k+1}) < \phi(v_k)$ contradicts the construction of $P$. Thus, $v_k$ is a descendant of $v_j$. Furthermore, since $i < k$, $v_i$ is pruned before $v_k$, so $v_k$ must be in a branch of $\mathcal{B}_{v_j}$ right of the branch containing $v_i$. Therefore, $p_k \geq p_i$. 

$(\Leftarrow):$ Suppose that $p_j < p_i$ and for any $k \in\N$ satisfying $i < k < j$, $p_k \geq p_i$. By the construction of $\mbox{L}\mathcal{D}(T)$, $p_j < p_i$ implies $\phi(v_j) < \phi(v_i)$. 
Moreover, the second supposed condition similarly implies $\phi(v_j) < \phi(v_k)$ 
for each $v_k$. If $v_j$ is not an ancestor of $v_i$, then $p_k\geq p_i$ for all $i < k < j$ implies $v_j$ is also not an ancestor of each $v_k$. In this case, the removal of each $v_k$ leaves the out-degree of $v_j$ unchanged. Thus, $v_j$ having out-degree zero at the $j$th step of the construction of $P$ implies $v_j$ had out-degree zero at the $i$th step, contradicting the construction of $P$. 

Therefore, $v_j$ is above $v_i$. Suppose $v_j$ is not the parent of $v_i$ and there exists some $k$ so that $v_k$ is the parent of $v_i$ and subsequently below $v_j$. Clearly $i < k < j$, but $v_k$ the parent of $v_i$ and below $v_j$ imply $\phi(v_j) < \phi(v_k) < \phi(v_i)$, and subsequently $p_j < p_k < p_i$, a contradiction. Thus $(v_j, v_i) \in E(T)$.
\end{proof}

The following theorem is the main result of the present work and makes precise the VCPC conditions necessary to detect when one vertex-colored arborescence is isomorphic to a subtree of another. Recall that for a sequence $Y = (y_n)_{n=0}^N$, $s(Y)$ denotes the shape of $Y$. Intuition behind the statement was provided in Section \ref{Sec:SubtreeIso}. 

\begin{theorem}\label{Thm:SubtreeIsomorphism}
The arborescence $T'$ is vertex-colored isomorphic to a subtree of $T$ if and only if there exists a set of ascending indices $\{i_j\}_{j=0}^{n'-1}$ that satisfy the following properties. 
\begin{itemize}
\item $c_{i_j} = c'_j$ for each $0 \leq j\leq n'-1$, 
\item $s\left((p_{i_j})_{j=0}^{n'-2}\right) = s\left((p'_j)_{j=0}^{n'-2}\right)$, and 
\item if there exists $a$ and $b$ so that $a< b$, $p'_a > p'_b$, and for any $a < k < b$, $p'_a \leq p'_k$, then for any $\ell\in\N$ satisfying $i_a < \ell < i_b$, $p_{i_a} \leq p_\ell$. We shall refer to this as the \textbf{incident edge} property. 
\end{itemize}
\end{theorem}

\begin{proof}
$(\Rightarrow):$ Suppose $S$ is the subtree of $T$ to which $T'$ is isomorphic. Let $v_{i_{n'-1}}$ be the root of $S$, and define $\phi^* : V(S) \to \{0, \dotsc, |V(S)| - 1\}$ so that $\phi^*$ is the order of traversal of $V(S)$ given by $\mbox{L}\mathcal{D}(S)$, i.e., for $u, v\in V(S)$, $\phi^*(u) < \phi^*(v)$ if and only if $\phi(u) < \phi(v)$. 
Then, define $\psi: V(T') \to V(T)$ so that $\psi(v'_j) = (I \circ (\phi^*)^{-1} \circ \phi')(v'_j)$, where $I : V(S) \to V(T)$ is the natural inclusion map. Note that Proposition \ref{Prop:CompositionOfPhisIsIsomorphism} implies the map defined by $(\phi^*)^{-1} \circ \phi'$ is a vertex-colored isomorphism from $T'$ to $S$, so $\psi$ has the same properties. Define indices $\{i_j\}_{j=0}^{n'-2}$ so that $\psi(v'_j) = v_{i_j}$. By the definition of $v_{i_{n'-1}}$, since $r' \in V(T')$ satisfies $r' = v_{n'-1}$, $\psi(v'_j) = v_{i_j}$ is also satisfied for $j = n'-1$. If $P^*$ is the VCPC for $S$ with vertex labeling $\phi^*$, Theorem \ref{Prop:VertexColoredPruferIsomorphism} implies $P^* = P'$. Thus, $c_{i_j} = c'_j$ for each $0 \leq j\leq n'-1$, so the first stated condition is satisfied. Furthermore, since $\phi^*(u) < \phi^*(v)$ if and only if $\phi(u) < \phi(v)$, the second stated condition is also satisfied. 


Suppose now that there exists $a$ and $b$ so that $a< b$, $p'_a > p'_b$, and for any $a < k < b$, $p'_a \leq p'_k$. By Lemma \ref{Lem:PruferAdjacency}, $v'_a$ and $v'_b$ are adjacent in $T'$. Since $\psi$ is an isomorphism, $v_{i_a}$ and $v_{i_b}$ are adjacent in $S$, which also means they are adjacent  in $T$. Thus, by Lemma \ref{Lem:PruferAdjacency}, for any $\ell\in\N$ satisfying $i_a < \ell < i_b$, $p_\ell \geq p_{i_a}$, so the third stated condition is satisfied. 

$(\Leftarrow):$ Suppose now there exists a set of indices $\{i_j\}_{j=0}^{n'-1}$ that satisfy all properties given in the theorem statement. Define $S$ to be the subforest of $T$ containing exactly the edges removed at the $i_j$th steps for $0\leq j\leq n'-2$ of the construction of $P$. Then $|E(S)| = |E(T')|$. Let $\psi : V(T') \to V(S)$ so that $\psi(v'_j) = v_{i_j}$. From the first stated condition on the indices $\{i_j\}_{j=0}^{n'-1}$, $\psi$ is color preserving. 

Suppose $v'_a$ and $v'_b$ are adjacent in $T'$ for some choice of $a$ and $b$. Without loss of generality, assume $a < b$, meaning $v'_a$ is pruned from $T'$ before $v'_b$ in the construction of $P'$, so $v'_b$ is the parent of $v'_a$. By Lemma \ref{Lem:PruferAdjacency}, $p'_b < p'_a$ and for any $k'\in\N$ satisfying $a < k' < b$, $p'_{k'} \geq p'_a$. Thus, by the third stated condition on the indices $\{i_j\}_{j=0}^{n'-1}$, for any $\ell\in\N$ satisfying $i_a < \ell < i_b$, $p_\ell \geq p_{i_a}$. By the second stated condition on the indices $\{i_j\}_{j=0}^{n'-1}$, $p'_b < p'_a$ implies $p_{i_b} < p_{i_a}$, so Lemma \ref{Lem:PruferAdjacency} implies $v_{i_a}$ and $v_{i_b}$ are adjacent, meaning $\psi$ preserves adjacency. Since $|E(T')| = |E(S)|$, $\psi$ also preserves non-adjacency, so $\psi$ provides an isomorphism from $T'$ to $S$ as vertex-colored arborescences.
%
\end{proof}


\subsection{Vertex-colored Arborescence Subisomorphism Algorithms}

The conditions on $\{i_j\}_{j=0}^{n'-1}$ given in Theorem \ref{Thm:SubtreeIsomorphism} imply the following algorithm to determine if $T'$ is isomorphic (as vertex-colored arborescence) to a subtree of $T$. Clearly the unique vertex labeling defined by $\mbox{L}\mathcal{D}$ and the VCPC from Definition \ref{Def:ColoredPruferCodes} must first be computed for any trees for which the subtree relationship will be determined. Determining vertex-colored arborescence isomorphism via VCPCs is designed with a large (at least a few thousand) corpus of trees in mind. 
The problem is to determine the entire partially ordered set $(\mathcal{T}_{\mathbb{N}}, <)$, where ``$<$'' denotes the subtree relationship as vertex-colored arborescences. Transitivity of a partial order makes it unnecessary to check all $\binom{\mathcal{T}_{\mathbb{N}}}{2}$ pairs of distinct trees, since if $T_1, T_2 \in \mathcal{T}_{\mathbb{N}}$ satisfying $T_1 < T_2$, then $T_1 < T_3$ for any $T_3 \in \mathcal{T}_{\mathbb{N}}$ satisfying $T_2 < T_3$ (and the same is true if every ``$<$'' is replaced with ``$>$''). While it depends on the given collection of trees, we expect few pairs can be omitted. 

Thus, we consider the work of computing $\mbox{L}\mathcal{D}$ and $P$ based on Definition \ref{Def:ColoredPruferCodes} one-time work completed ahead of time for each tree, and we are only interested in the computational complexity of determining if $T'$ is a subtree of $T$ already given VCPCs $P'$ and $P$. The following algorithms provide one way to determine this relationship given the conditions provided by Theorem \ref{Thm:SubtreeIsomorphism}. 
\begin{enumerate}
\item Determine subsets of $P$ whose second row matches the second row of $P'$. More formally, determine which increasing sets of indices $\{i_j\}_{j=0}^{n'-1}$ are ``color preserving'', i.e., $c_{i_j} = c'_j$ for each $0 \leq j < n'$. 

\begin{algorithm}[H]
\caption{Determine subsets of $P$ whose second row matches the second row of $P'$}
\begin{algorithmic}
\REQUIRE $P \gets$ an array containing two elements, namely the arrays of the two rows of the Pr\"{u}fer code. This is given as input to the function. 
\REQUIRE $C' \gets $ an array containing arrays of the form $[c'_i, i]$ for each $0 \leq i <n'$. \newline
\STATE $\mathcal{S} \gets$ the set containing the zeroth element of $C'$. 
\FOR{$k$ from zero to the length of $P$ minus one}
\STATE $new\_arrays \gets$ the arrays in $\mathcal{S}$ whose last element is an array (called $x$), and the zeroth element of $x$ is the color in the $k$th column of $P$. 
\STATE amend each array $x$ in $new\_arrays$ by replacing the last element of $x$ with $k$ and (if $len(x) < n'$) appending the entry of $C'$ at index $len(x)$. 
\STATE $\mathcal{S} \gets$ the union of $\mathcal{S}$ and $new\_arrays$
\ENDFOR
\STATE restrict $\mathcal{S}$ to only contains arrays where the last element is an integer (these are exactly the arrays in $\mathcal{S}$ of length $n'$)
\RETURN $\mathcal{S}$
\end{algorithmic}
\end{algorithm}



\item Given the collection $\mathcal{S}$ of color-preserving index sets provided by the last item, determine which of them are shape preserving, i.e., those that satisfy the second stated condition in Theorem \ref{Thm:SubtreeIsomorphism}. Algorithm \ref{Alg:ShapeProp} computes the ``shape'' of an array of nonnegative integers. 
\begin{algorithm}[H]
\caption{Computing the shape of an array of nonnegative integers}\label{Alg:ShapeProp}
\begin{algorithmic}
\REQUIRE $P \gets$ an array containing nonnegative integers. \newline
\STATE $dist\_elems \gets $ the array of distinct elements in $P$ ordered ascendingly.
\STATE $shape \gets$ an array containing the index in $dist\_elems$ for each $p$ in $P$.  
\RETURN $shape$
\end{algorithmic}
\end{algorithm}

The ``shape'' of the first row of $P'$ and the first row of $P$ with columns indexed by elements of $\mathcal{S}$ are computed. Since each of these inputs are arrays of length $n'$, each execution of the algorithm provided seems to have time complexity given by $O(n')$. Together these applications of the provided function require $O(|\mathcal{S}|\cdot n')$ time, but then each of the reduced shapes of slices of $P$ must be compared to the reduced shape of $P'$. Overall, this step requires $O(|\mathcal{S}| \cdot n')$ time. While $|\mathcal{S}|$ can vary drastically, on average we expect $|\mathcal{S}|$ to be quite small relative to $n'$. Moreover, $\mathcal{S}$ only shrinks during this step of the algorithm, so we expect $\mathcal{S}$ to be even smaller when considered as input to the last step (presented next).

\item Lastly, reduce $\mathcal{S}$ to contain only those index arrays $s$ for which the reduced shape of $P'$ and the reduced shape of $P$ indexed by $s$ match. It only remains to check index sets of $\mathcal{S}$ for the edge incidence condition in Theorem \ref{Thm:SubtreeIsomorphism}. We accomplish this in Algorithm \ref{Alg:EdgeIncidenceProp}, where we pass the first row of $P'$, the first row of $P$, and an index array $s$. 
\begin{algorithm}[H]
\caption{Determine if an index set passes the incident edge property in Theorem \ref{Thm:SubtreeIsomorphism}}\label{Alg:EdgeIncidenceProp}
\begin{algorithmic}
\REQUIRE $P \gets$ the first row of $P$
\REQUIRE $P' \gets$ the first row of $P'$
\REQUIRE $\mathcal{S} \gets$ the set of index arrays provided as input 
\STATE
\STATE $pairs \gets []$
\STATE $start \gets 0$
\FOR{$i$ in the indices of $P'$}
\IF{$P[i] < P'[start]$}
\STATE append the array $[start, i]$ to $pairs$
\STATE $start \gets i$
\ENDIF
\ENDFOR
\STATE
\FOR{$x$ in $\mathcal{S}$}
\FOR{$interval$ in $pairs$}
\STATE $P\_slice \gets $ the slice of $P$ between the entries of interval (first index is inclusive and second is exclusive)
\IF{the minimum entry of $P\_slice$ is not the zeroth element}
\STATE break
\ENDIF
\ENDFOR
\RETURN True
\ENDFOR
\RETURN False
\end{algorithmic}
\end{algorithm}

The first loop iterates through the elements of $P'$, so has time complexity $O(n')$. The second loop iterates through arrays of valid indices stored in $\mathcal{S}$, then takes an appropriate slice of $P$ to determine if the first element is the minimum of all elements in the slice. An index set encoding the subtree relationship between $T$ and $T'$ will satisfy this property (the first element of the slice of $P$ is the minimum for the entire interval) for all intervals identified by the first \textit{for} loop. So, for any given iteration through the outer \textit{for} loop, either an interval is found that does not satisfy the property, in which case the inner \textit{for} loop is terminated and the next index array is checked, or all intervals satisfy the ``first element is the minimum'' property, and both loops are terminated, ending the execution of the entire function. 
\end{enumerate}

While we are unable to precisely describe the computational complexity of the above algorithm, Section \ref{Sec:PythonExp} provides insight into its performance relative to standard isomorphism algorithms already available (even including relatively small up-front costs of computing the canonical labeling and VCPC for each of a collection of trees). 

\subsection{Python Experiment}\label{Sec:PythonExp}

While we are unable to explicitly determine the computational complexity of the vertex-colored rooted tree isomorphism and subtree isomorphism algorithms stated above, we built a pythonic experiment to test run time against the standard networkx implementations in Python. In this experiment, we randomly generate a corpus of trees $\mathcal{T}$ in the following way. 
\begin{enumerate}
\item Let $m$ denote the largest possible order of a tree. 
\item Let $N$ denote the number of trees that will be generated. 
\item Let $C$ denote the maximum number of vertex colors. Let $\mathbb{N}_C = \{0, \dotsc, C-1\}$. 
\item For each of the $N$ trees, generate a uniformly random integer in the interval $[1,m]$, called $n$. 
\item Then call the \texttt{networkx.random\_tree} method with $n$ as the number of vertices (interestingly enough, this tree is constructed by uniformly at random generating a Pr\"{u}fer code of the appropriate length then reconstructing the tree). This returns an undirected tree with vertices labeled $\{0,\dotsc,n-1\}$. Call this tree $T'$.
\item Choose a random tuple from $(\mathbb{N}_C)^n$ and assign the vertices of $T'$ to have colors given by the tuple. This is achieved through the \texttt{ networkx .set\_node\_attributes} method. 
\item Define $T$ to be a breadth-first search of $T'$ rooted at vertex $0$ (obtained through the \texttt{networkx.bfs\_tree} method). 
\item Form the collection $\mathcal{T}$ by repeating steps (1) through (7) $N$ times. After each iteration, add $T$ from step (7) to $\mathcal{T}$. 
\end{enumerate}

For the Pr\"{u}fer code implementation, determining the corpus partition by isomorphism class requires the following three steps. 
\begin{enumerate}
\item The canonical ordering of each tree must be computed. 
\item The vertex colored Pr\"{u}fer code of each tree must be computed. 
\item Then pairwise code comparisons can be used to determine isomorphism. 
\end{enumerate}

On the other hand, the standard implementation in networkx requires just one step, namely checking isomorphism. This can be done through the \texttt{networkx.is\_isomorphic} method, where a node match is required to ensure vertex colors are also checked. Clearly there are three parameters for any trial, namely values of $m$, $N$, and $C$. The following table summarizes the ratio of run times between the two implementations. It is important to note that the pairwise comparisons to determine isomorphisms algorithm is identical for both the Pr\"{u}fer code and standard python implementations, so the exact same sequence of isomorphism evaluations are performed in both. Moreover, the partition into isomorphism classes by both methods are compared afterwards to check for equality, mitigating logic errors in the code writing. 

Table \ref{Tab:Iso} also summarizes the ratio of run times to compute the partial order on $\mathcal{T}$ given by the (vertex-colored) subtree relation. As in the isomorphism implementation, the only difference between the networkx and Pr\"{u}fer code implementations is which subtree isomorphism test is used. For the Pr\"{u}fer implementation, the three algorithms in the previous section are implemented. For the networkx version, the \\ \texttt{networkx.algorithms.isomorphism.GraphMatcher.subgraph\_is\_isomorphic} method is used to determine subtree isomorphism. All ratios given are the networkx implementation time divided by the Pr\"{u}fer time. Note that the time required to compute canonical vertex orderings and VCPCs of all trees are included in the isomorphism time (first column of ratios as described by the previous paragraph) not the subtree partial order time (second column as described by this paragraph).

\begin{table}[H]
\centering
\begin{tabular}{c  c  c | c c}
$m$ & $N$ & $C$ & Isomorphism time ratio & Subtree Partial Order time ratio  \\
\hline
8 & 1000 & 4 & 21.22 & 4.61 \\
 &  & 5 & 18.14 & 4.77 \\
 &  & 6 & 22.06 & 4.98 \\
 &  & 7 & 22.7 & 5.02 \\
\hline
 & 5000 & 4 & 75.78 & 4.27 \\
 &  & 5 & 78.83 & 4.54 \\
 &  & 6 & 80.3 & 4.72 \\
 &  & 7 & 85.76 & 4.66 \\
\hline
 & 10000 & 4 & 145.83 & 4.16\\
 &  & 5 & 148.51 & 4.14 \\
 &  & 6 & 154.08 & 4.22\\
 &  & 7 & 164.62 & 4.21\\
\hline
12 & 1000 & 4 & 10.57 & 1.91 \\
 &  & 5 & 10.89 & 2.67 \\
 &  & 6 & 11.16 & 3.38 \\
 &  & 7 & 11.19 & 3.93 \\
\hline
 & 5000 & 4 & 43.74 & 1.94 \\
 &  & 5 & 46.14 & 2.59 \\
 &  & 6 & 47.74 & 3.31 \\
 & & 7 & 43.77 & 3.8 \\
\hline
 & 10000 & 4 & 83.49 & 1.75 \\
 &  & 5 & 83.21 & 2.43\\
 &  & 6 & 92.54 & 3.07\\
 &  & 7 & 85.84 & 3.56\\
\end{tabular}
\caption{Ratio (networkx implementation time divided by Pr\"{u}fer implementation time) of run times to (1) compute the corpus partition into vertex-colored isomorphism classes and (2) compute the partial order given by subtree isomorphism. Note that the Pr\"{u}fer code isomorphism implementation time also includes the one-time work of computing the canonical ordering and VCPC for each tree.}\label{Tab:Iso}
\end{table}

\section{Future Directions}\label{Sec:FutureDirections}

There are a number of ways the subgraph isomorphism test provided by VCPCs could be extended to more general settings. Work of \cite{YanWan2020} has already extended standard Pr\"{u}fer codes to the setting of labeled graphs. Recall that the problem considered in Section \ref{Sec:CompClassify} studies arborescences, i.e., rooted trees where every edge is directed ``away'' from the root. Clearly if arborescence $T'$ is a subtree of $T$, then the underlying undirected tree of $T'$ is a subtree of the undirected tree underlying $T$. However, this sufficient condition is not necessary. Thus, tackling the full subtree isomorphism problem for vertex-colored trees via VCPCs requires consideration of multiple rootings of undirected trees $T$ and $T'$. The following result suggests multiple rootings are adequate to address the full problem, although the algorithm implied by the result seems computationally impractical. 

\begin{prop}
Let $T$ and $T'$ be vertex-colored trees. Then $T'$ is isomorphic (as vertex-colored trees) to a subtree of $T$ if and only if there exists an arborescence of each of $T$ and $T'$ so that their VCPCs satisfy the properties of Theorem \ref{Thm:SubtreeIsomorphism}. 
\end{prop}
\begin{proof}
$(\Rightarrow)$: First suppose $T'$ is isomorphic as vertex-colored trees to a subtree of $T$. Let $\psi : V(T') \to V(T)$ denote the vertex-colored isomorphism between $T'$ and the subtree of $T$ induced on $\psi(V(T'))$. Arbitrarily choose leaf $\ell' \in V(T')$. Then let $A_{T'}$ be the arborescence of $T'$ rooted at $\ell'$. 

Consider $\psi(\ell') \in V(T)$. If $\psi(\ell')$ is a leaf in $T$, define $\ell = \psi(\ell')$. Otherwise, $\ell'$ has degree one in $T'$ and $\psi(\ell')$ has degree more than one in $T$. Thus, there exists leaf $\ell \in V(T)$ so that if $Q$ is the path in $T$ with endpoints $\psi(\ell')$ and $\ell$, then $V(Q) \cap \psi(V(T')) = \{\psi(\ell')\}$. 

In either case, let $A_T$ be the arborescence of $T$ rooted at $\ell$. In this way, $A_{T'}$ is isomorphic to a subarborescence of $A_T$, completing the proof in this direction. 

$(\Leftarrow)$: Conversely suppose there exists an arborescence of each of $T$ and $T'$ so that their VCPCs satisfy the properties of Theorem \ref{Thm:SubtreeIsomorphism}. Suppose these arborescences are $A_T$ and $A_{T'}$ respectively. Then Theorem \ref{Thm:SubtreeIsomorphism} implies $A_{T'}$ is isomorphic to a subarborescence of $A_T$. Let $\psi : V(A_{T'}) \to V(A_T)$ be such an isomorphism. If $i_T : V(A_T) \to V(T)$ and $i_{T'} : V(A_{T'}) \to V(T')$ are the natural inclusion maps, then $i_T \circ \psi \circ i_{T'}^{-1} : V(T') \to V(T)$ is an isomorphism between $T'$ and a subtree of $T$, completing the proof. 
\end{proof}

A priori, the vertex-colored subtree isomorphism $\psi : V(T') \to V(T)$ is unknown. The following algorithm provides a way to use VCPCs to answer the subtree isomorphism problem for vertex-colored trees. 
\begin{enumerate}
\item Fix a leaf $\ell'$ of $T'$ and let $A_{\ell'}$ be the arborescence of $T'$ rooted at $\ell'$. Compute $P'$. 
\item For each leaf $\ell \in V(T)$ compute VCPC $P_\ell$ of $A_\ell$, the arborescence of $T$ rooted at $\ell$. 
\item Determine if $P_\ell$ and $P'$ satisfy Theorem \ref{Thm:SubtreeIsomorphism}. 
\item If yes, then $T'$ is isomorphic to a subtree of $T$. Return True.
\item If not, proceed to the next iteration of the loop given by step (2). 
\item If the loop terminates, then the test failed for every choice of leaf in $T$. Thus, $T'$ is not isomorphic to a vertex-colored subtree of $T$, so return False. 
\end{enumerate}
The computational complexity suggested by this algorithm seems to far exceed that of the algorithm in Section \ref{Sec:CompClassify}, to the point of infeasibility at any scale. Is it possible to improve either upon this algorithm or on the definition of VCPCs to allow for a computationally feasible subtree isomorphism test for general undirected vertex-colored trees?

\bibliographystyle{plain}
\bibliography{ref}

\begin{thebibliography}{10}

\bibitem{AbbBacHanWilZam2015}
Amir Abboud, Arturs Backurs, Thomas~Dueholm Hansen, Virginia~Vassilevska Williams, and Or~Zamir.
\newblock Subtree isomorphism revisited.
\newblock {\em ACM Transactions on Algorithms}, 14(3), 2018.

\bibitem{Awe1989}
Baruch Awerbuch, Andrew Goldberg, Michael Luby, and Serge Plotkin.
\newblock Network decomposition and locality in distributed computation.
\newblock {\em 30th Annual Symposium on Foundations of Computer Science}, pages 364--369, 1989.

\bibitem{bab}
L\'{a}szl\'{o} Babai.
\newblock Graph isomorphism in quasipolynomial time [extended abstract].
\newblock In {\em Proceedings of the Forty-Eighth Annual ACM Symposium on Theory of Computing}, STOC '16, page 684–697, New York, NY, USA, 2016. Association for Computing Machinery.

\bibitem{Bus1997}
Samuel~R. Buss.
\newblock Alogtime algorithms for tree isomorphism, comparison, and canonization.
\newblock In {\em Computational Logic and Proof Theory}, pages 18--33. Springer Berlin Heidelberg, 1997.

\bibitem{CamPet2005}
Saverio Caminiti and Rossella Petreschi.
\newblock String codings of trees with locality and heritability.
\newblock {\em Proceedings of the 11th International Conference on Computing and Combinatorics (COCOON '05)}, 3595:251--262, 08 2005.

\bibitem{Cay1889}
Arthur Cayley.
\newblock A theorem on trees.
\newblock {\em The Quarterly Journal of Pure and Applied Math}, 23:376--378, 1889.

\bibitem{ChoKimSeoShi2004}
Manwon Cho, Dongsu Kim, Seunghyun Seo, and Heesung Shin.
\newblock Colored pr\"{u}fer codes for k-edge colored trees.
\newblock {\em Electr. J. Comb.}, 11, 07 2004.

\bibitem{ChuGra1981}
F.R.K. Chung and R.L. Graham.
\newblock Recent results in graph decompositions.
\newblock {\em Combinatorics, London Mathematical Society Lecture Note Series}, 52:103--123, 1981.

\bibitem{DarHarPhuPro2018}
R.~W.~R. Darling, David~G. Harris, Dev~R. Phulara, and John~A. Proos.
\newblock The combinatorial data fusion problem in conflicted-supervised learning.
\newblock arxiv:1809.08723, 2018.

\bibitem{DeoMic2002}
N.~Deo and Paulius Micikevicius.
\newblock A new encoding for labeled trees employing a stack and a queue.
\newblock {\em Bulletin of the Institute of Combinatorics and its Applications}, 34, 01 2002.

\bibitem{Gal1967}
Tibor Gallai.
\newblock Transitiv orientierbare graphen.
\newblock {\em Acta Mathematica Academiae Scientiarum Hungarica}, 18:25--66, 1967.

\bibitem{Luc1882}
Edouard Lucas.
\newblock {\em Recreations Mathematiques}.
\newblock Gautheir-Villars, 1882-1894.

\bibitem{Moo1970}
J.W. Moon.
\newblock {\em Counting Labeled Trees}.
\newblock Canadian Mathematical Monographs. 1970.

\bibitem{Nev1953}
E.~H. Neville.
\newblock The codifying of tree-structure.
\newblock {\em Mathematical Proceedings of the Cambridge Philosophical Society}, 49(3):381--385, 1953.

\bibitem{FreightData}
U.S.~Department of~Transportation; Bureau~of Transportation Statistics (BTS); Federal Highway Administration~(FHWA).
\newblock Freight analysis framework faf5.6.1 2018-2023.
\newblock 2024.

\bibitem{Pru1918}
Heinz Pr\"{u}fer.
\newblock Neuer beweis eines satzes \"{u}ber permutationen.
\newblock {\em Archiv der Mathematischen Physik}, 27:742--744, 1918.

\bibitem{RobSey1983}
Neil Robertson and P.~D. Seymour.
\newblock Graph minors. i. excluding a forest.
\newblock {\em Journal of Combinatorial Theory, Series B}, 35(1):39--61, 1983.

\bibitem{ShaTsu1999}
Ron Shamir and Dekel Tsur.
\newblock Faster subtree isomorphism.
\newblock {\em Journal of Algorithms}, 33(2):267--280, 1999.

\bibitem{Sta1999}
Richard~P. Stanley.
\newblock {\em Enumerative Combinatorics, vol. 2}.
\newblock Cambridge Studies in Advanced Mathematics. Cambridge University Press, 1999.

\bibitem{Veb1912}
Oswald Veblen.
\newblock An application of modular equations in analysis situs.
\newblock {\em Annals of Mathematics}, 14(1):86--94, 1912.

\bibitem{WanWanWu2009}
X.~Wang, L.~Wang, and Y.~Wu.
\newblock An optimal algorithm for prufer codes.
\newblock {\em Journal of Software Engineering and Applications}, 2:111--115, 2009.

\bibitem{Wil2024}
Virginia~Vassilevska Williams, Yinxuan Xu, Zixuan Xu, , and Renfei Zhou.
\newblock New bounds for matrix multiplication: from alpha to omega.
\newblock In {\em Proceedings of the 2024 Annual ACM-SIAM Symposium on Discrete Algorithms (SODA)}, pages 3792--3835. SIAM, 2024.

\bibitem{YanWan2020}
Lin Yang and Yongjie Wang.
\newblock Prufer coding: A vectorization method for undirected labeled graphs.
\newblock {\em IEEE Access}, 8:175360--175369, 2020.

\end{thebibliography}

\end{document}